\documentclass{article}

\usepackage{arxiv}
\usepackage[round]{natbib}

\usepackage{subfigure}

\usepackage[utf8]{inputenc} 
\usepackage[T1]{fontenc}    
\usepackage{mathtools}
\usepackage{hyperref}       
\mathtoolsset{showonlyrefs=true}

\usepackage{url}            
\usepackage{booktabs}       
\usepackage{amsfonts}       
\usepackage{nicefrac}       
\usepackage{microtype}      
\usepackage{xcolor}         
\usepackage{mathtools}

\usepackage{epstopdf,amsmath}
\usepackage{bbm}
\usepackage{authblk}

\usepackage{ amssymb }
\usepackage{mathrsfs}
\usepackage{makecell}

\usepackage[normalem]{ulem}
\allowdisplaybreaks

\usepackage{amsmath}    
\usepackage{algorithm}
\usepackage{algpseudocode}
\usepackage{colortbl}
\usepackage{amsmath,color}

\usepackage{algpseudocode}
\usepackage{graphicx}
\usepackage{epstopdf}
\usepackage{pifont}

\usepackage{multirow}
\allowdisplaybreaks

\DeclareMathOperator*{\argmax}{argmax}
\DeclareMathOperator*{\argmin}{argmin}

\newcommand{\abs}[1]{\left\lvert #1 \right\rvert}
\DeclareMathOperator*{\iid}{i.i.d.}
\newcommand{\norm}[1]{\left\lVert #1 \right\rVert}
\newcommand{\ceil}[1]{\lceil #1 \rceil}

\newcommand{\expec}[2]{\mathbb{E}_{#2}\left[ #1 \right] }
\newcommand\numberthis{\addtocounter{equation}{1}\tag{\theequation}}  

\def\fn[#1]#2{{f_{#1}\left(x_{#2}\right)}}

\newtheorem{claim}{Claim}[section]

\newtheorem{theorem}{Theorem}[section]
\newtheorem{proposition}{Proposition}[section]
\newtheorem{assumption}{Assumption}[section]
\newtheorem{lemma}{Lemma}[section]
\newtheorem{remark}{Remark}

\def\L{{Lipschitz }}
\providecommand{\mc}{\mathcal}

\def\exp{{\rm exp}}

\def\cD{{\cal D}}

\def\cF{{\cal F}}
\def\cK{{\cal K}}

\def\tzeta{{\tilde \zeta}}
\def\tV{{\tilde V}}

\def\tz{{\tilde z}}

\def\zh{{H_t\left( x_t,u_t\right) }}

\def\t{\theta}
\def\m{\mu}

\def\zh{z_{k+1/2}}
\def\kh{{k+1/2}}
\def\subop{\mathcal{G}}
\def\xikh{\xi(z_k,w_{k+1})}
\def\xik1{\xi_{k+1}(z_\kh,w_{k+1})}
\def\etak1{\eta_{k+1}}
\def\id{\mathbf{I}}

\newcommand\redsout{\bgroup\markoverwith{\textcolor{red}{\rule[0.5ex]{2pt}{2.4pt}}}\ULon}

\newenvironment{talign*}
 {\csname align*\endcsname}
 {\endalign}

\title{A Central Limit Theorem for Algorithmic Estimator of Saddle Point }%

\author{%
  Abhishek Roy, Yi-An Ma \\
  Halıcıoğlu Data Science Institute\\
  University of California, San Diego\\
  \texttt{\{a2roy, yianma \}@ucsd.edu} \\
}

\begin{document}
%

%
\maketitle
\begin{abstract}
In this work, we study the asymptotic randomness of an algorithmic estimator of the saddle point of a globally convex-concave and locally strongly-convex strongly-concave objective. Specifically, we show that the averaged iterates of a Stochastic Extra-Gradient (SEG) method for a Saddle Point Problem (SPP) converges almost surely to the saddle point and follows a Central Limit Theorem (CLT) with optimal covariance under martingale-difference noise and the state(decision)-dependent Markov noise. To ensure the stability of the algorithm dynamics under the state-dependent Markov noise, we propose a variant of SEG with truncated varying sets. Interestingly, we show that a state-dependent Markovian data sequence can cause Stochastic Gradient Descent Ascent (SGDA) to diverge even if the target objective is strongly-convex strongly-concave. The main novelty of this work is establishing a CLT for SEG for a stochastic SPP, especially under sate-dependent Markov noise. This is the first step towards online inference of SPP with numerous potential applications including games, robust strategic classification, and reinforcement learning. We illustrate our results through numerical experiments. 
\end{abstract}
\vspace{-0.0in}
\section{Introduction}
\vspace{-0.05in}
A central theme in statistical learning and inference is understanding the distribution of the algorithmic estimators.
For empirical risk minimizers, there is an entire field devoted to studying their asymptotic distributions under various regimes and model classes~\citep[c.f.,][to list a few]{Huber73annals,Portnoy85annals,Sur19pnas}.
For the stochastic gradient descent (SGD) algorithm and its variants that provide approximate minimizers, stochastic approximation literature provides a framework for such understanding~\citep{kushner2012stochastic,benveniste2012adaptive,ruppert1988efficient,polyak1992acceleration,liang2010trajectory,su2014differential,lei2020variance}.
The upshots of these results are that the SGD iterates converge to the global minima of a strongly convex function almost surely, and 
follow a Central Limit Theorem (CLT), with their average achieving the optimal covariance~\citep[c.f.,][for a comprehensive survey]{li2022revisiting}.

However, for the stochastic min-max optimization methods, there has been a paucity of such results.
Despite its empirical success and widespread application in algorithmic game theory and economics \citep{shapiro2002minimax,mertens1981stochastic}, Generative Adversarial Network (GAN) training \citep{goodfellow2014generative}, robust optimization \citep{vasile2014solution}, and reinforcement learning \citep{jin2020efficiently}, little has been known about the asymptotic distribution of the algorithmic estimators of saddle points.
This is due in part to the complex and potentially unstable nature of the saddle point optimization methods.
To this end, very recently, researchers have established the CLT of the linear stochastic approximation for quadratic min-max games \cite{mou2020linear}, and linearly constrained stochastic approximation \cite{yan2022stochastic} assuming $\iid$ data. It is well-known that Stochastic Gradient Descent Ascent (SGDA), which is equivalent to classical stochastic approximation algorithm, runs into convergence issues for convex-concave objective \cite{abernethy2021last,liang2019interaction,yang2020global}
To resolve this, several efficient algorithms including the stochastic extra-gradient \cite{nemirovski2004prox,korpelevich1976extragradient} (SEG) \citep{mokhtari2020unified}, and the optimistic gradient descent-ascent (OGDA) method \citep{liang2019interaction,daskalakis2017training,rakhlin2013online} have been proposed to solve min-max optimization. SEG is known to approximate proximal point method which has better stability properties and is fundamentally different from SGD or SGDA \cite{toulis2021proximal}. Motivated by this, we adopt the Stochastic Extra-Gradient (SEG) method as the algorithmic framework to solve the following stochastic min-max optimization problem, also known as the Saddle Point Problem (SPP).
\begin{align*}
    \textstyle(\theta^*,\mu^*)\coloneqq\underset{\t\in\mathbb{R}^{d_{\t}}}{\argmin}~\underset{\m\in\mathbb{R}^{d_{\m}}}{\argmax}~\expec{F(\t,\m,w)}{}
    =\underset{\t\in\mathbb{R}^{d_{\t}}}{\argmin}~\underset{\m\in\mathbb{R}^{d_{\m}}}{\argmax}~f(\t,\m),\numberthis\label{eq:mainprob}
\end{align*}
where $f(\theta,\mu)$ is a globally convex-concave but locally strongly-convex strongly-concave function, $w\in\mathbb{R}^{d_{\t}+d_{\m}}$ is the data sample on which the function is evaluated, and the expectation is taken with respect to $w$. In this paper, we focus on the online setting and ask the following question,
\vspace{-0.05in}
\begin{quote}
Does a CLT with optimal covariance hold for the algorithmic estimator, the averaged iterates of SEG, of the saddle point of a min-max optimization, especially when the data is state(decision)-dependent Markovian?
\end{quote}
\vspace{-0.05in}
SPP has been studied in both min-max optimization \citep{antonakopoulos2020adaptive,diakonikolas2022potential,mukherjee2020decentralized} and Variational Inequality Problem (VIP) literature \citep{korpelevichextragradient,tseng1995linear}. 
Existing works have only focused on the convergence of the SEG iterates to $(\t^*,\mu^*)$ in expectation under various settings \citep{mishchenko2020revisiting,beznosikov2021distributed,gorbunov2022stochastic,tiapkin2022stochastic}. CLT results for stochastic proximal point methods have been established under martingale-difference noise, which appears with $\iid$ data sampling, in \citep{toulis2021proximal,asi2019stochastic}. \cite{toulis2021proximal} establishes the CLT for the proximal point iterates instead of the averaged iterates and hence does not achieve optimal covariance. \cite{asi2019stochastic} establishes CLT for the averaged iterates of the \textit{exact} update of proximal point method for stochastic minimization problem under similar settings as \cite{polyak1992acceleration}. 
In contrast, we show that the \textit{averaged} iterates $(\bar{\t}_k,\bar{\mu}_k)$ of SEG converges almost surely to $(\t^*,\mu^*)$, and follows a CLT for \textit{both} martingale-difference noise and more interestingly, state-dependent Markov noise associated with state(decision)-dependent Markovian data sampling. This paves the way to construct confidence intervals for $(\t^*,\mu^*)$ and facilitates inference tasks.

\textbf{Novelties} 
In Section~\ref{sec:martdiff}, we establish the CLT for a martingale-difference noise sequence (Theorem~\ref{th:martdiffclt}). In this setting, even though the proof follows similar techniques as in \cite{polyak1992acceleration}, a few non-trivial changes are required. While results in \cite{polyak1992acceleration} apply to the stochastic approximation method for minimization of a strongly convex function (Assumption 4.1), we study SEG, an approximation of proximal point method, for convex-concave SPP where stochastic approximation (SGDA) may not even converge. Compared to SGDA, SEG introduces additional error terms in the convergence analysis which one needs to control. Since we are solving SPP instead of a minimization problem, unlike \cite[][Theorem 3]{ polyak1992acceleration}, the function-value suboptimality cannot be chosen as Lyapunov function. Instead, we use $\mc G(\theta,\mu)\coloneqq f(\theta,\mu^*)-f(\theta^*,\mu)$ as the suboptimality measure and the distance from the optimizer as our Lyapunov function.  

In Section~\ref{sec:markov}, we prove a CLT (Theorem~\ref{th:marclt}) in the more general setting of state(decision)-dependent Markovian data where the transition kernel of the data depends on the iterates. This noise setting is motivated by applications where the data distribution is decision-dependent, e.g., robust strategic classification, multitask strategic classification and relative cost maximization in competitive markets \citep{wood2022stochastic}. Expected convergence has been shown by \cite{levanon2021strategic} in this setting. State-dependent Markovian data should be contrasted with state-independent Markovian data setting in stochastic minimization problems, where suitable mixing conditions are often assumed implying noise samples separated by a fixed time are nearly independent \citep{nagaraj2020least,duchi2012ergodic,wang2022stability}. We can not make such an assumption because of the dependence of the noise on the iterates. It is particularly challenging to establish limit theorems under this noise since SEG is not guaranteed to be stable under state-dependent Markov noise \cite{iusem2017extragradient}. We propose and analyze a variation of SEG, Truncated SEG (TSEG), which we prove to be stable under state-dependent noise. We show almost sure convergence and establish a CLT for the algorithm dynamics. Establishing CLT in this setting requires different techniques from \cite{polyak1992acceleration}. Specifically, after showing stability, the proof involves a different noise decomposition, analyzing an auxiliary dynamics, and showing that this auxiliary dynamics is close to the original algorithm. 

The summary of our \textbf{main contributions} are:

\begin{enumerate}
\item Towards characterizing the asymptotic randomness of an algorithmic estimator of saddle point of a convex-concave (locally strongly-convex strongly-concave) objective under martingale-difference noise, we show that the averaged iterates of SEG converge almost surely to $(\t^*,\m^*)$, follows a CLT and achieves optimal asymptotic covariance.
\item We propose a truncated varying SEG for SPP under state(decision)-dependent Markovian data to ensure stability. We establish similar results as the martingale-difference noise setting in the state-dependent Markovian data setting. 
\end{enumerate}
\vspace{-0.0in}
\section{Related Work}\label{sec:relwork}
\vspace{-0in}
\textbf{Convergence in expectation of SPP} Convergence in expectation has been proved for SEG and its variants under various settings: Nash equilibrium of stochastic bilinear games has been studied in \citep{gidel2019negative,gidel2018variational,li2022convergence,golowich2020last}; Minibatch-based SEG methods have been analyzed in \citep{iusem2017extragradient,bot2019forward,mishchenko2020revisiting}. A two-stepsize version of SEG has been proposed and analyzed in \citep{hsieh2020explore} to show almost sure convergence to stationary point of monotone games. SPP with decision-dependent data distribution has recently been considered in \citep{wood2022stochastic}. We go beyond the convergence in expectation and characterize the asymptotic distribution of the averaged iterates of SEG. Moreover, \citep{wood2022stochastic} assumes the noisy gradient is either unbiased or has an exponentially decaying tail implying the bias is small with high-probability. We make no such assumption. \\
\textbf{Inference for SPP} \citep{minsker2020asymptotic} proves a CLT for the exact saddle point of a min-max objective in the context of offline robust risk minimization. In contrast to exact saddle point, we study a stochastic algorithmic estimator of saddle point in an online setting. In a different setting, \citep{holcblat2022empirical} proposes a saddle point estimator which is defined as the exact maximizer of a criterion and establishes a CLT for the estimator without considering any optimization algorithm. \citep{liu2022confidence} studies constructing confidence region for SPP in an offline setting and does not explicitly consider any optimization algorithm. The two most relevant works to us are \citep{mou2020linear}, and \citep{fort2015central}. \citep{mou2020linear} shows that the averaged iterates of a Gradient Descent Ascent algorithm with $\iid$ noise follows a CLT for a quadratic min-max game. \citep{fort2015central} analyzes classical stochastic approximation under state-dependent Markovian noise. As we have discussed above, stochastic approximation may not converge for convex-concave objective whereas we analyze SEG, an approximation of stochastic proximal point method which provably converges. Apart from this significant fundamental difference, \citep{fort2015central} assumes that only a finite number of truncations are needed for truncated stochastic approximation (Assumption A1). On the contrary, instead of assuming this, we explicitly prove this for TSEG in Theorem~\ref{th:asconvmar} and the proof is quite challenging as we explain in Section~\ref{sec:markov}. \citep{fort2015central} assumes the gradients to be bounded which fails in the strongly-convex strongly-concave case. We do not assume the gradients to be bounded. \citep{fort2015central} only establishes the CLT for the iterates (Section 4, Proposition 4.1) of stochastic approximation which does not achieve the optimal covariance. In contrast, our results are for the averaged iterates of SEG which achieves optimal covariance.  \\
\textbf{Other relevant literature} State-dependent Markov noise sequence has been studied before across several fields, for example, stochastic approximation \citep{andrieu2005stability,liang2010trajectory,roy2023online}, nonconvex minimization \citep{royconstrained}, strategic classification~\citep{cai2015optimum,hardt2016strategic,mendler2020stochastic} and reinforcement learning~\citep{bartlett1992learning, qu2020finite, goldberg2013adaptive, zhang2021statistical}. There has been a plethora of works on non-stochastic EG method \cite{diakonikolas2017accelerated,mokhtari2020unified}, and designing algorithms and proving convergence in expectation to various notions of saddle points for nonconvex nonconcave SPP; to name a few  \citep{adolphs2018non,xu2023unified,diakonikolas2021efficient}. Since this is an independent line of work with very little relevance to the problem presented here, we will not discuss this in further detail. 
\vspace{-0.0in}
\section{Motivating Example}\label{sec:motvn}
\vspace{-0.0in}
Examples of martingale-difference noise are abundant throughout stochastic SPP literature \citep[c.f.][for a non-exhaustive list]{iusem2017extragradient,bot2019forward,mishchenko2020revisiting,beznosikov2023stochastic,loizou2021stochastic,fallah2020optimal}. State-dependent Markov noise has been studied before in stochastic minimization across various applications: strategic classification with adaptive best response \cite{li2022state}; policy-gradient \cite{karimi2019non}, and actor-critic algorithm in reinforcement learning \cite{wu2020finite}. In SPP, it has only recently been studied in \cite{wood2022stochastic}. So here we provide an example of convex-concave (locally strongly-convex strongly-concave) saddle point problem with state-dependent Markov noise -- relative profit maximization in a competitive market for electric vehicle charging. Consider two competing service providers A and B for Electric Vehicle charging aiming to maximize their relative profits over $N$ zones.  At each zone $i$ there is an average baseline price $p_i$. Provider A sets the price differential $\theta_i$ to charge per minute. For provider A, let $a_{i,k}$ be the demand in zone $i$ at time $k$. The revenue, and the zone-based utility cost are given by $a_{i,k}(\theta_i+p_i)$, and $r_i(\theta_i+p_i)$ respectively. Quality of service cost that penalizes large deviation from baseline for zone $i$ is $\gamma_{A,i}(\theta)\theta_i^2$, $\gamma_{A,i}(\theta)\geq0$ with $\gamma_{A,i}(\theta)>0$ in some arbitrarily small neighborhood around the equilibrium point. Similarly, for provider B in zone $i$ at time $k$, the demand, revenue, zone-based utility cost, and quality of service cost are given by $b_{i,k}$, $b_{i,k}(\mu_i+p_i)$, $r_i(\mu_i+p_i)$, $\gamma_{B,i}\mu_i^2$, $\gamma_{B,i}(\mu)\geq0$ with $\gamma_{B,i}(\mu)>0$ in some arbitrarily small neighborhood around the equilibrium point, respectively.
The distribution of $(a_k,b_k)$ is shown to linearly depend on $(\theta_{k-1},\mu_{k-1})$ \cite{wood2022stochastic} as well as the demand at the previous time step $(a_{k-1},b_{k-1})$ \cite{yan2022many,zhang2020spatiotemporal}. So the demands are modeled to be sampled from, 
\begin{align*}
a_{k+1}\sim D_A+\rho a_{k}+A_1\theta_k+A_2\mu_k, \quad b_{k+1}\sim D_B+\rho b_{k}+B_1\theta_k+B_2\mu_k,
\end{align*}
where elements of $A_1$($A_2$) are negative(positive), $B_1=A_2$, $B_2=A_1$, $0<\rho<1$, $D_A,D_B$ are $\iid$ samples from some prior distribution for which the $2+\epsilon$ moments exist for arbitrarily small $\epsilon>0$. As explained in \cite{wood2022stochastic}, the objective of the optimization is independent of $p_i$. So, removing $p_i$, the optimization problem can be written in the following way which is a strongly-convex strongly-concave SPP with state-dependent Markov noise. 
\begin{align*}
    \min_{\theta\in \mathbb{R}^n} \max_{\mu\in \mathbb{R}^n} \mathbb{E}[\norm{\Gamma_A(\theta) \theta}_2^2-\norm{\Gamma_B(\mu) \mu}_2^2
    -\theta^\top (a+r)
    +\mu^\top (b+r)],\numberthis\label{eq:evcharge}
\end{align*}
where $\Gamma_A(\theta)=diag\{\gamma_A^i(\theta)\}$, $\Gamma_B(\mu)=diag\{\gamma_B^i(\mu)\}$, 
and the expectation is w.r.t the stationary distribution of $a_{\theta^*}$ and $b_{\mu^*}$. In the above display we omit the subscript $i$ from the variables to denote the vector, e.g., $a_{k}=[a_{1,k},a_{2,k},\cdots,a_{n,k}]$. 
\vspace{-0.0in}
\section{Assumptions and Notations}
\vspace{-0.05in}
Let $z\coloneqq (\theta,\mu)$, $z^*\coloneqq (\theta^*,\mu^*)$, and $H(z)\coloneqq \begin{bmatrix} \nabla_{\t} f(\theta,\mu)^\top,& - \nabla_{\m} f(\theta,\mu)^\top\end{bmatrix}^\top$. Let $H(z,w)$ be the noisy evaluation of $H(z)$ with data sample $w$. 
Let us denote the noise by $\xi(z,w)\coloneqq H(z)-H(z,w)$. Let $\cF_k$ denote the natural filtration consisting of $\{z_0,w_1,z_1,\cdots,z_k\}$. We will use $C>0$ to denote any constant which does not depend on $k$. We use $x\lesssim y$ to denote $x\leq Cy$ where $C>0$ is a constant. Let \begin{align*}
    Q^*=\begin{bmatrix}
        \nabla^2_{\t\t}f(\t^*,\mu^*)& \nabla^2_{\mu\t}f(\t^*,\mu^*)\\
       - \nabla^2_{\t\mu}f(\t^*,\mu^*) & -\nabla^2_{\mu\mu}f(\t^*,\mu^*)
    \end{bmatrix}.
\end{align*}
Note that $Q^*$ is not the Hessian of $f(\theta,\mu)$. Now, we state our assumptions below.
\begin{assumption}\label{as:strongcon}
Let $f(\t,\mu)$ be twice differentiable, and convex in $\theta$ and concave in $\mu$. Let $Q^*$ be Hurwitz matrix, i.e., the real parts of the eigen values of $Q$ are positive. There exists an arbitrarily small neighborhood $\mc Z$ of $z^*$, such that for $z\in \mc Z$, we have,
\begin{align*}
\norm{H(z)-Q^*(z-z^*)}_2\lesssim \norm{z-z^*}_2^2,
\end{align*} 
i.e., in $\mc Z$, $f(\theta, \mu)$ is
strongly-convex in $\theta$, and strongly-concave in $\mu$, i.e., $\exists m>0$ a constant such that,
\begin{align}
    H(z)^\top(z-z^*)\geq m \norm{z-z^*}_2^2 \quad z\in\mc Z. \label{eq:strongcon}
\end{align}
\end{assumption}
\begin{assumption}\label{as:lipgrad}
Let $\mathcal{V}:\mathbb{R}^{d_\t+d_\mu}\to [1,\infty)$ be a function such that there is a constant $C_\mathcal{V}>0$ with $\expec{\mathcal{V}(w)^{\alpha_0}}{}\leq C_\mathcal{V}$ for some $\alpha_0>2$. For all $ w$, $H(z,w)$ is Lipschitz continuous, i.e., $\exists L_G>0$ such that
\begin{align*}
\norm{H(z_1,w)-H(z_2,w)}_2\leq L_G\norm{z_1-z_2}_2\mathcal{V}(w).
\end{align*}
\end{assumption}
\begin{assumption}\label{as:noise}
The noise sequence $\{\xi(z_k,w_{k+1})\}_k$ is a martingale difference sequence, i.e., $
    \expec{\xi(z_k,w_{k+1})|\cF_k}{}=0, 
$
and has locally bounded variance, i.e., for some constant $C>0$, 
\begin{align*}
    &\expec{\norm{\xi(z_k,w_{k+1})}_2^2}{}\leq C(1+\norm{z_k-z^*}_2^2), \\ 
    &\expec{\norm{\xi(z_\kh,w_{k+1})}_2^2}{}\leq C(1+\norm{\zh-z^*}_2^2).\numberthis\label{eq:noisevar}
\end{align*}
\end{assumption}
\begin{assumption}\label{as:asymcovar}
We have the following limit.
$
    \expec{H(z^*,w_k)H(z^*,w_k)^\top|\cF_{k-1}}{}\overset{P}{\to}\Sigma.
$
\end{assumption}
\begin{assumption}\label{as:lindeberg}
As $\mathcal{C}\to\infty$,
$$\sup_k\expec{\norm{H(z^*,w_{k})}_2^2\mathbbm{1}\left(\norm{H(z^*,w_{k})}_2>\mathcal{C}\right)|\cF_{k-1}}{}\overset{P}{\to}0.$$
\end{assumption}
\begin{remark}
Assumption~\ref{as:strongcon}-\ref{as:noise} is common in stochastic SPP literature; see \cite{mishchenko2020revisiting,gorbunov2022stochastic} and the references therein. Assumption~\ref{as:asymcovar}-\ref{as:lindeberg} are mainly used in CLT literature. Assumption~\ref{as:strongcon} simply states that $f(\theta,\mu)$ is locally strongly-convex strongly-concave function in an arbitrarily small neighborhood of $z^*$. In Assumption~\ref{as:lipgrad}, we assume $H(z,w)$ is locally Lipschitz continuous where the Lipschitz coefficient may depend on some function of $w$, $\mathcal{V}(w)\geq 1$. \eqref{eq:noisevar} in Assumption~\ref{as:noise} implies that the variance of the noise is only locally bounded which is weaker than the uniformly bounded variance assumption in SPP as pointed out in \cite{mishchenko2020revisiting}. Combining Assumption~\ref{as:lipgrad}, and \eqref{eq:noisevar}, one has,
\begin{align*}
    \expec{\norm{H(z_k,w_{k+1})}_2^2|\cF_k}{}
    \leq & C(1+\norm{z_k-z^*}_2^2),\numberthis\label{eq:noisegradvar}
\end{align*}
for some constant $C>0$, i.e., a similar bound as \eqref{eq:noisevar} holds for the noisy gradient too. Assumption~\ref{as:asymcovar} states that the noisy gradient sequence has an asymptotic covariance conditioned on $\cF_{k-1}$. Indeed, without the existence of covariance, asymptotic normality cannot hold. Assumption~\ref{as:lindeberg} is similar to the so-called Lindeberg's condition, widely used in CLT literature \cite{clarke1980probability}, which controls the number of variables with high variance in the sequence $\{H(z^*,w_k)\}_k$.
\end{remark}
\vspace{-0.0in}
\section{Martingale-difference Noise}\label{sec:martdiff}
\vspace{-0.0in}
In this section, we state SEG and results for the martingale-difference noise sequence. In Algorithm~\ref{alg:seg}, in the extrapolation step \eqref{eq:halfupd} we first compute the extrapolated point $z_\kh$. In the update step \eqref{eq:fullupd}, we compute the update $z_{k+1}$ using the gradient evaluated at $z_\kh$. We assume that, at each iteration, the data sample is shared between the two steps \cite{gorbunov2022stochastic,li2022convergence,mishchenko2020revisiting}. Note that, this assumption holds for most machine learning applications like bi-linear games, GAN, and adversarial training \cite{li2022convergence}. In fact, there are examples where SEG diverges when independent samples are used for steps \eqref{eq:halfupd} and \eqref{eq:fullupd} \cite{li2022convergence}. Intuitively, this happens because with independent samples, SEG is no longer an approximation of proximal point method.

We first show in Theorem~\ref{th:asconv}, that $z_k$, and $\bar{z}_k$ converge to $z^*$ almost surely. Then in Theorem~\ref{th:martdiffclt}, we show that $\sqrt{n}(\bar{z}_n-z^*)$ is asymptotically normal and $\bar{z}_n$ achieves asymptotically optimal covariance. 
\begin{algorithm}[t]
	\caption{Stochastic Extra-Gradient (SEG) Method} \label{alg:seg}
	\textbf{Input:} $\eta_k$, $\theta_0\in \mathbb{R}^{d_{\theta}}$, $\mu_0\in \mathbb{R}^{d_{\mu}}$
	\begin{algorithmic}[1]
		\State \textbf{for} $k=1,\cdots,n$ \textbf{do}
		\State \textbf{Update} 
		\begin{align*}
		    &\zh=z_k-\eta_{k+1}H(z_k,w_{k+1}) \numberthis\label{eq:halfupd}\\
		    &z_{k+1}=z_{k}-\eta_{k+1}H(\zh,w_{k+1})\numberthis\label{eq:fullupd}
		\end{align*}
		\State \textbf{end for}
	\end{algorithmic}	
 \textbf{Output:} $\bar{z}_n=\frac1n\sum_{k=1}^nz_k$
\end{algorithm}
\begin{theorem}\label{th:asconv}
Let Assumption~\ref{as:strongcon}-\ref{as:noise} be true. Then, choosing $\eta_k=\eta k^{-a}$ for sufficiently small $\eta$, and $1/2<a<1$, for the iterates of Algorithm~\ref{alg:seg}, we have, 
\begin{align*}
    z_k\overset{a.s.}{\to}z^*, \qquad \text{and}\qquad \bar{z}_k\overset{a.s.}{\to}z^*.\numberthis\label{eq:asconv}
\end{align*}
\end{theorem}
The proof of Theorem~\ref{th:asconv} is in Appendix~\ref{pf:asconv}.
\begin{theorem}\label{th:martdiffclt}
Let Assumption~\ref{as:strongcon}-\ref{as:lindeberg} be true. Then, choosing $\eta_k=\eta k^{-a}$ for sufficiently small constant $\eta$, and $1/2<a<1$, for the iterates of Algorithm~\ref{alg:seg}, the following holds, 
\begin{align*}
\sqrt{n}\left(\bar{z}_n-z^*\right)\overset{d}{\to}N\left(0,{Q^*}^{-1}\Sigma {Q^*}^{-1}\right),
\end{align*}
where $\Sigma$ is defined in Assumption~\ref{as:asymcovar}.
In addition, this is the asymptotically optimal covariance. 
\end{theorem}
\textbf{Proof sketch} Recall that $\subop(\theta,\mu)=f(\theta,\mu^*)-f(\theta^*,\mu)$. First, we show that, 
\begin{align*}
     \expec{\norm{z_{k+1}-z^*}_2^2|\cF_k}{}\leq 
     \left(1+C'\etak1^2\right)\norm{z_{k}-z^*}_2^2
     -2\etak1\subop(z_k)+C'\eta_{k+1}^2.
\end{align*}
Then the almost sure convergence follows from Robbins-Siegmund theorem and Kronecker's lemma. Next, we establish the CLT result for linear gradients of the form $H(z_k)=Qz_k$ for some Hurwitz matrix $Q$. Define $V_k\coloneqq z_k-z^*$, $\bar{V}_{k}=\sum_{i=1}^kV_i/k$, and for $ j>i$, $Y_i^j\coloneqq \prod_{k=i+1}^j(\id-\eta_k Q)$, $Y_i^i\coloneqq\id$, $i=1,2,\cdots$. Then one gets the following decomposition. 
\begin{align*}
\sqrt{k}\bar{V}_{k}=&\underbrace{\textstyle\frac{1}{\sqrt{k}}\sum_{j=1}^kY_0^k(z_0-z^*)}_{T_1}
-\underbrace{\textstyle\frac{1}{\sqrt{k}}\sum_{i=1}^k\sum_{j=1}^{i}\eta_jY_j^i(H(z_{j-1/2},w_j)-H(z_{j-1},w_j))}_{T_2}\\
    &+\underbrace{\textstyle\frac{1}{\sqrt{k}}\sum_{j=1}^kQ^{-1}\xi(z_{j-1},w_{j})}_{I_1}
    +\underbrace{\textstyle\frac{1}{\sqrt{k}}\sum_{j=1}^k\sum_{i=j}^k(\eta_j-\eta_{i+1})Y_j^i\xi(z_{j-1},w_{j})}_{I_2}
    -\underbrace{\textstyle\frac{1}{\sqrt{k}}\sum_{j=1}^kQ^{-1}Y_j^{k+1}\xi(z_{j-1},w_{j})}_{I_3}
\end{align*}
We show that $I_1\overset{d}{\to}N\left(0,{Q}^{-1}\Sigma {Q}^{-1}\right)$, and $\expec{\norm{T_1}_2^2}{}$, $\expec{\norm{T_2}_2^2}{}$, $\expec{\norm{I_2}_2^2}{}$, and $\expec{\norm{I_3}_2^2}{}$ converge to 0. Finally, we show that this linear dynamics and Algorithm~\ref{alg:seg} are asymptotically equivalent when $Q=Q^*$.
\section{State Dependent Markov Noise}\label{sec:markov}
\vspace{-0.05in}
In this section, we focus on the state-dependent Markovian data where the transition probability kernel $P_{z_k}$ at time $k$ depends on the current state $z_k$. In this setting, any iterative algorithm inherently induces time-varying distribution on $w_k$, and the SPP becomes ill-posed as defined in \eqref{eq:mainprob}. Instead, we look for an equilibrium point $z^*\coloneqq(\t^*,\m^*)$ defined as the saddle point of $\expec{F(\t,\mu,w)}{}$ when the expectation is taken with respect to the stationary distribution (when it exists) $\pi_{(\t^*,\m^*)}$ of the Markov chain, i.e., we redefine $(\t^*,\m^*)$ in the following manner.  
\begin{align*}
z^*=(\t^*,\m^*)\coloneqq \argmin_{\t\in\mathbb{R}^{d_\t}}\argmax_{\m\in\mathbb{R}^{d_\m}}\expec{F(\t,\m,w)}{w\sim\pi_{(\t,\m)}}.\numberthis\label{eq:equi}
\end{align*}
Of course, the problem is only sensible if an equilibrium point exists. Here we will concentrate on such a case. Note that, when the data sequence $\{w_k\}$ is state-dependent Markovian, so is the induced noise $\xi(z_k,w_{k+1})$. Below we state our assumptions. 

Let $z^*\in\mc Z$ where $\mc Z\subset\mathbb{R}^{d_\theta+d_\mu}$ is an arbitrary open set. For any mapping $G:\mathbb{R}^{d_\t+d_\mu}\to\mathbb{R}^{d_\t+d_\mu}$ define the following norm with respect to a function $\mathcal{V}:\mathbb{R}^{d_\t+d_\mu}\to [1,\infty)$: $\norm{G}_\mathcal{V}=\underset{w\in\mathscr{W}}{\sup}\left(\norm{G(w)}_2/\mathcal{V}(w)\right)$, and let $L_\mathcal{\mathcal{V}}=\{G:\mathbb{R}^{d_\t+d_\mu}\to\mathbb{R}^{d_\t+d_\mu},\sup_{w\in\mathscr{W}}\norm{g}_\mathcal{V}<\infty\}$. 
\begin{assumption}\label{as:noisemar}
Let $\{w_k\}_k$ be a Markov chain controlled by $\{z_{k-1}\}_k$, i.e., there exists a transition probability kernel $P_{z_k}(\cdot,\cdot)$ such that
\begin{align*}
    \mathbb{P}(w_{k+1}\in B|z_0,w_0,\cdots,z_k,w_k)=P_{z_k}(w_k,B),
\end{align*}
almost surely for any Borel-measurable set $B\subseteq\mathbb{R}^{d_\t+d_\mu}$ for $k\geq 0$. For any $z\in \mathbb{R}^{d_\t+d_\mu}$, $P_z$ is irreducible and aperiodic. Additionally, there exists a function $\mathcal{V}:\mathbb{R}^{d_\t+d_\mu}\to [1,\infty)$ and a constant $\alpha_0> 2$ such that for any compact set $Z'\subset \mc Z$, we have the following.
\begin{enumerate}
    \item \label{eq:asa31} There exist a set $\mathcal{C}_1\subset \mathbb{R}^{d_\t+d_\mu}$, an integer $l$, constants $0<\lambda<1$, $b$, $\kappa$, $\delta>0$, and a probability measure $v$ such that,
    \begin{align*}
        &\sup_{z\in Z'}P_z^l\mathcal{V}^{\alpha_0}(w)\leq \lambda \mathcal{V}^{\alpha_0}(w)+bI(w\in {\mc C}_1)~\forall w\in\mathbb{R}^{d_\t+d_\mu},\\
        &\sup_{z\in Z'}P_z \mathcal{V}^{\alpha_0}(w)\leq\kappa \mathcal{V}^{\alpha_0}(w)\quad \forall w\in\mathbb{R}^{d_\t+d_\mu},\\
        &\inf_{z\in Z'}P_z^l(w,A)\geq \delta v(A) \quad \forall w\in {\mc C}_1, \forall A\in \mathcal{B}_{\mathbb{R}^{d_\t+d_\mu}},
    \end{align*}
    where $\mathcal{B}_{\mathbb{R}^{d_\t+d_\mu}}$ is the Borel $\sigma$-algebra over $\mathbb{R}^{d_\t+d_\mu}$, and for a function $q$, $P_z q(w) \coloneqq \int P_z (w,y) q(y)\, dy $.
    \item \label{eq:asa32} There exists a constant $C>0$, such that, for all $w\in\mathbb{R}^{d_\t+d_\mu}$,
    $
        \sup_{z\in Z'}\norm{H(z,w)}_\mathcal{V}\leq C.$
    \item \label{eq:asa33} There exists a constant $C>0$, such that, for all $(z,z')\in Z'\times Z'$,
    \begin{align*}
        &\norm{P_z G-P_{z'} G}_\mathcal{V}\leq C\norm{G}_\mathcal{V}\norm{z-z'}_2\ \forall G\in L_\mathcal{V}\\
        & \norm{P_z G-P_{z'} G}_{\mathcal{V}^{\alpha_0}}\leq C\norm{G}_{\mathcal{V}^{\alpha_0}}\norm{z-z'}_2\ \forall G\in L_{\mathcal{V}^{\alpha_0}}.
    \end{align*}
\end{enumerate}
\end{assumption}
\begin{assumption}\label{as:stepsize}
$\{\eta_k\}_k$ and $\{d_k\}_k$ sequences in Algorithm~\ref{alg:tseg} satisfy the following conditions.
\begin{align*}
&\sum_{k=1}^\infty \eta_k=\infty, \ \lim_{k\to\infty}(k\eta_k)=\infty, \ \frac{\eta_{k+1}-\eta_{k}}{\eta_k}=o(\eta_{k+1}),\\
&d_k\leq C\eta_k^{\frac{1+\varepsilon}{2}}, \ \sum_{k=1}^\infty\frac{\eta_k^{\frac{1+\varepsilon}{2}}}{\sqrt{k}}<\infty, \ \sum_{k=1}^\infty(\eta_kd_k+\left(\frac{\eta_k}{d_k}\right)^{\alpha_0})<\infty,
\end{align*}
for some $\varepsilon\in(0,1)$, and $\alpha_0$ is as in Assumption~\ref{as:noisemar}.
\end{assumption}
\begin{remark}[Comments on the assumptions]
Condition \ref{eq:asa31} of Assumption~\ref{as:noisemar} is the so-called drift condition and $\mathcal{V}$ is the drift function commonly used in the Markov chain literature; see \cite{meyn2012markov} for a book-level treatment. Intuitively, it implies that for each fixed $z\in\mathcal{Z}$, the data sequence is $\mathcal{V}^{\alpha_0}$-uniformly ergodic, i.e., for each fixed $z\in\mathcal{Z}$, mixes exponentially fast to $\pi_z$. When $H(z,w)$ is bounded, one can choose $\mathcal{V}(w)=1$. Condition \ref{eq:asa32} of Assumption~\ref{as:noisemar} implies that in the compact set $\mathcal{Z}'$, $\norm{H(z,w)}_2$ is $O(\mathcal{V}(w))$. Note that, setting $z_2=z^*$ in Assumption~\ref{as:lipgrad}, one readily sees that Condition \ref{eq:asa32} is only slightly stronger than commonly used Assumption~\ref{as:lipgrad} since Condition \ref{eq:asa32} is assumed only over the compact set $\mathcal{Z}'$. Condition \ref{eq:asa33} controls the change of transition kernel by imposing a Lipszhitz property on $P_z$ wr.t. $z$. The main implication of Assumption~\ref{as:noisemar} is that it ensures the existence and regularity of a solution $u(z,w)$ to Poisson equation of the transition kernel $P_z$ given by $u(z,w)-P_{z}u(z,w)=H(z,w)-H(z)$. The solution of Poisson equation has been fundamental in the analysis of additive functionals of Markov chain (see \citep{andrieu2005stability,meyn2012markov,douc2018markov} for details). Assumption~\ref{as:noisemar} has been verified for numerous applications \cite{li2022state,karimi2019non,wu2020finite,liang2010trajectory}. Assumption~\ref{as:stepsize} imposes a stricter condition on the stepsizes. For example, when $\alpha_0>2(1+\varepsilon)^2/(1-\varepsilon)$ for $0<\varepsilon<1$, one can choose, $\eta_k=Ck^{-a}$, and $d_k=Ck^{-b}$ where $a=1/(1+\varepsilon)$, $b=1/2$.
\end{remark}
\vspace{-0.0in}
\begin{remark}
\textbf{(Divergence of SGDA for state-dependent Markovian data)}
To understand the effect of state dependence of the noise on SGDA consider the following setting. Let $d_\t=d_\m=1$, and $
        F(\t_0,\m_0,w_1)=\t\m+\left(\t^2-\m^2\right)/2+[1 \quad 1]^\top w_1,
    $
    where, $w_{k+1}=\mathscr{R}(\norm{z_k}_2)[-\norm{\t_k}_2^2/2\quad \norm{\m_k}_2^2/2]^\top+\varrho_k,$ $\mathscr{R}(\cdot):\mathbb{R}\to\mathbb{R}$ is any smooth, Lipschitz continuous function such that $|\mathscr{R}(y)|\leq 1$ for any $y\in\mathbb{R}$, $\mathscr{R}(y)=1$ for $y\geq 1$, $\mathscr{R}(y)=0$ when $y\leq 1/2$, and  ${\varrho_{k}}\sim N(0, I_2)$ is an $\iid$ sequence independent of filtration $\mathcal{F}_{k-1}$ for $k=0,1,2,\cdots$. Note that for a fixed $z_k=z$, $w_{k+1}\sim N(\mathscr{R}(\norm{z}_2)[-\norm{\t}_2^2/2\quad \norm{\m}_2^2/2]^\top,I_2)=\pi_z$, i.e., conditioned on $z$, $w_{k+1}$ is distributed as $\pi_z$, the stationary distribution, automatically satisfying the drift condition as the drift condition implies exponential mixing to stationary distribution (Condition 1 of Assumption~\ref{as:noisemar}). Condition 2 of Assumption~\ref{as:noisemar} is satisfied with $\mathcal{V}(w_{k+1})=1+\norm{w_{k+1}}_2^2$ on any compact set $Z'$. Condition 3 of Assumption~\ref{as:noisemar} is implied by boundedness and Lipschitz continuity of $\mathscr{R}$, and compactness of $Z'$. In this setting, if SGDA is initialized at any point $z_0$ where $\|z_0\|\geq 1$, then, $f(\theta_0,\mu_0)=\theta\mu$. Then, SGDA updates will spiral out from $z_0$, i.e., 
    \begin{align*}
        \expec{\norm{z_1}_2^2|\cF_0}{}=(1+\eta_1^2)\norm{z_0}_2^2+2\eta_1^2.
    \end{align*}
    Intuitively, because of the state-dependence of the noise, $f(\theta_0,\mu_0)$ becomes a bilinear function which is locally strongly-convex strongly-concave around $z^*=(0,0)$ for which SGDA is known to diverge when initiated outside the strongly-convex strongly-concave neighborhood \cite{mokhtari2020unified,yang2020global}. Whereas SEG converges almost surely to $z^*$ and a CLT holds for the averaged iterates of SEG as we show below in Theorem~\ref{th:asconvmar}, and Theorem~\ref{th:marclt} 
\end{remark}
To show the almost sure convergence of Algorithm~\ref{alg:seg} for the martingale-difference noise sequence a step in the analysis involves an interaction term between the iterates and the noise of the form $\expec{(z_k-z^*)^\top \xi(z_k,w_{k+1})|\cF_k}{}$ \eqref{eq:interaction}. This term vanishes since $\xi(z_k,w_{k+1})$ is a martingale-difference sequence which is crucial for the stability of the dynamics \cite{iusem2017extragradient}. For state-dependent Markov noise, we do not have this property anymore. Moreover, because of the dependence on $z_k$, the chain is not exponentially mixing either. But to establish any limit theorem, we need to establish the stability of the algorithm dynamics. So, we modify Algorithm~\ref{alg:seg} in a manner similar to varying truncation stochastic approximation algorithm used in \citep{liang2010trajectory} and propose Algorithm~\ref{alg:tseg}. In this algorithm we maintain sequence of compact sets $\{\mathcal{K}_q\}_q$, which we call truncation sets, such that 
\begin{align*}
\mathcal{K}_q\subset \text{int}(\mathcal{K}_{q+1}),\qquad \text{and}\qquad \cup_{q\geq 0} \mathcal{K}_q=\mc Z,
\end{align*}
where $\text{int}(\cdot)$ denote the interior of a set. One potential choice for $\{\mathcal{K}_q\}$ is a sequence of balls with increasing radii. Additionally, we maintain a decreasing sequence of thresholds $\boldsymbol{d}=\{d_k\}_k$. At each iteration $k$, first an iterate $z_{k+1}$ is generated from $z_k$ using the vanilla SEG step as in Algorithm~\ref{alg:seg}. Then, if $z_{k+1}\notin \mathcal{K}_q$ or the change in the consecutive iterates is bigger than a predefined threshold, i.e., $\norm{z_{k+1}-z_k}_2\geq d_k$, we reinitialize the algorithm from $z_0$ with a bigger truncation set $\mathcal{K}_{q+1}$. Note that we do not need the prior knowledge of any compact set containing $z^*$, and consequently, we do not use any explicit projection operator. We show that Algorithm~\ref{alg:tseg} naturally identifies some compact set containing $z^*$ in which the dynamics remain confined (Theorem~\ref{th:finitetrunc}). Proving Theorem~\ref{th:finitetrunc} is quite involved and one of the factors that sets this analysis apart from that of martingale-difference noise. Then we establish the almost sure convergence and CLT results. In contrast to the proof of Theorem~\ref{th:martdiffclt}, this part of the proof involves a different noise decomposition, analyzing an auxiliary dynamics, and showing that this auxiliary dynamics is close to the original algorithm. 

\textbf{Proof sketch of stability} To prove the stability of Algorithm~\ref{alg:tseg}, we first show that under Assumption~\ref{as:noisemar}, the number of truncations needed in Algorithm~\ref{alg:tseg} is finite almost surely. Then the dynamics of Algorithm~\ref{alg:tseg} remain confined in the compact set $\mathcal{K}_{\varkappa_t+1}$. Then we show that $z_k$ converges to $z^*$ almost surely (Theorem~\ref{th:asconvmar}).
\begin{algorithm}[t]
	\caption{Truncated Stochastic Extra-Gradient (TSEG) Method} \label{alg:tseg}
	\textbf{Input:} $\boldsymbol{d}=\{d_k\}_k$, $\boldsymbol{\eta}=\{\eta_k\}_k$, $\theta_0\in \mathbb{R}^{d_{\theta}}$, $\mu_0\in \mathbb{R}^{d_{\mu}}$, $\varkappa_1=0$
	\begin{algorithmic}[1]
		\State \textbf{for} $k=0,\cdots n$ \textbf{do}
        \State \textbf{ sample } $w_{k+1}\sim P_{z_k}(\cdot)$
		\State \textbf{Update} 
		\begin{align*}
		    &\zh=z_k-\eta_{k+1}H(z_k,w_{k+1}) \numberthis\label{eq:halfupdmar}\\
		    &z_{k+1}=z_{k}-\eta_{k+1}H(\zh,w_{k+1})\numberthis\label{eq:fullupdmar}\\
      &\varkappa_{k+1}=\varkappa_k
		\end{align*}
  \State \textbf{if } $\left(\norm{z_{k+1}-z_k}_2\geq d_k\text{ OR } z_{k+1}\notin \mathcal{K}_{\varkappa_k}\right)$
\State $z_{k+1}=z_1$, $w_k=w_1$, and $\varkappa_{k+1}=\varkappa_k+1$.
  \State \textbf{end if}
		\State \textbf{end for}
	\end{algorithmic}	
 \textbf{Output:} $\bar{z}_n=\frac1n\sum_{k=1}^nz_k$
\end{algorithm}
\begin{theorem}\label{th:finitetrunc}
Let Assumption~\ref{as:strongcon}, Assumption~\ref{as:noisemar}, and Assumption~\ref{as:stepsize} be true. Then, for Algorithm~\ref{alg:tseg}, almost surely there exists an $\mc M$ such that $k_{\varkappa_{\mc M}}<\infty$ and $k_{\varkappa_{{\mc M}}+1}=\infty$, where $k_{\varkappa_{\mc M}}$ denote the iteration at which the $\mc M$th truncation happens (Theorem~\ref{th:finitetrunc}).
\end{theorem}
\begin{proof}[Proof of Theorem~\ref{th:finitetrunc}]\label{pf:th:finitetrunc}
First note that irrespective of what $\mathcal{K}_0$ is, since $\{\mathcal{K}_q\}_q$ is an increasing cover of $\mc Z$, there exists a finite $q_0$ such that $z^*\in\mathcal{K}_{q_0}$. 

Let $\mathbb{P}_{z_0,w_1}(A)$ denote the probability of an event $A$ defined on the Markov chain with initial conditions $z_0,w_1$. Let $T_k$ denote the number of times reinitialization has occurred by time $k$. 
We claim that the tail probability of the number of reinitializations decrease exponentially, i.e., we have the following.
\begin{claim}\label{cl:tailprob}
There exists a constant $0<\rho<1$, and some constant $C>0$, such that, for any positive integer $m_0$, we have $~~\mathbb{P}_{z_0,w_1}\left(\sup_{k\geq 1}T_k\geq m_0\right)\leq C\rho^{m_0}.$
\end{claim}
It is easy to see that Claim~\ref{cl:tailprob} implies Theorem~\ref{th:finitetrunc} by Borel-Cantelli lemma. We defer the proof of Claim~\ref{cl:tailprob} to Appendix~\ref{pf:tailprob}.
\end{proof}
Since almost surely no truncation is needed after $\varkappa_m<\infty$, it is sufficient to establish the asymptotic properties of $z_k$ in the compact set $\mathcal{K}_{\varkappa_m}$. In addition, Assumption~\ref{as:noisemar} guarantees the existence of solution $u(z,w)$ to the Poisson equation. These two are key factors that imply that the noise $\xi({z_{k-1},w_k})$ can be decomposed into a martingale difference sequence, a telescopic-sum like term, and a term with norm comparable to step sizes \citep{liang2010trajectory} given by Lemma~\ref{lm:noisedecompbound}. 
\begin{lemma}[Informal]\label{lm:noisedecompbound}
Let Assumption~\ref{as:strongcon}, Assumption~\ref{as:noisemar}, and Assumption~\ref{as:stepsize} be true. Then the following decomposition takes place
\begin{align*}
    \xi(z_{k-1},w_{k})=e_{k}+\nu_{k}+\zeta_{k},
\end{align*}
where, $\{e_k\}$ is martingale difference sequence, $\expec{\norm{\nu_k}_2}{}\leq \eta_{k}$, and $\zeta_k=(\tilde{\zeta}_k-\tilde{\zeta}_{k+1})/\eta_k$, 
where $\expec{\|\tilde{\zeta}_k\|_2}{}\leq \eta_k$. Moreover, when $z_k\overset{a.s.}{\to} z^*$, one has,
 $
     \frac{1}{\sqrt{n}}\sum_{k=1}^ne_ke_k^\top\overset{d}{\to}N(0,\Sigma_s),
 $
 where $\Sigma_s=\lim_{k\to\infty}e_ke_k^\top$. 
\end{lemma}
Now we state the following results on almost sure convergence and CLT for Algorithm~\ref{alg:tseg} in the state-dependent Markov noise setting. 
\begin{theorem}\label{th:asconvmar}
    Let Assumptions~\ref{as:strongcon}, ~\ref{as:noisemar}, and ~\ref{as:stepsize} be true. Then for Algorithm~\ref{alg:tseg} we get,
        $z_k\overset{a.s.}{\to}z^*.$
\end{theorem}
The proof of Theorem~\ref{th:asconvmar} is in Appendix~\ref{pf:asconvmar}.
\begin{theorem}\label{th:marclt}
    Let Assumptions~\ref{as:strongcon},~\ref{as:noisemar}, and ~\ref{as:stepsize} be true. Then for Algorithm~\ref{alg:tseg} we get,
    \begin{align*}
        \sqrt{n}\left(\bar{z}_n-z^*\right)\overset{d}{\to}N\left(0,{Q^*}^{-1}\Sigma_s{Q^*}^{-1}\right),
    \end{align*}
    where $\Sigma_s=\lim_{k\to\infty}e_ke_k^\top$.
     and the asymptotic covariance is optimal. 
\end{theorem}
The proof of Theorem~\ref{th:marclt} is in Appendix~\ref{pf:marclt}.\\
\textbf{Proof sketch of Theorem~\ref{th:asconvmar}, and \ref{th:marclt}:} Define the following perturbed sequence.
\begin{align*}
    \tz_{k+1}=z_{k+1}+\tzeta_{k+2}
    =\tz_k-\eta_{k+1} H(\tz_k)+\eta_{k+1}(e_{k+1}+\tau_{k+1}),
\end{align*}
where $\tau_{k+1}=\nu_{k+1}+\tilde\xi(z_k,w_{k+1})+\omega_{k+1}$, and $\omega_{k+1}=H(\tz_k)-H(z_k)$. Using Assumption~\ref{as:lipgrad}, and Lemma~\ref{lm:noisedecomposition}, we have that $\tau_{k+1}$ is small, i.e., $
    \expec{\norm{\tau_{k+1}}_2^2}{}
    \lesssim \eta_{k+1}^2.
$
We first prove the almost sure convergence result and CLT for this perturbed sequence. Next we show that $\{\tilde{z}_k\}_k$ and $\{z_k\}_k$ are asymptotically equivalent. 
\section{Experiments}
\begin{figure}[t] 
    \centering
    \subfigure[]{\label{fig:martdiff}\includegraphics[width=80mm,height=38mm]{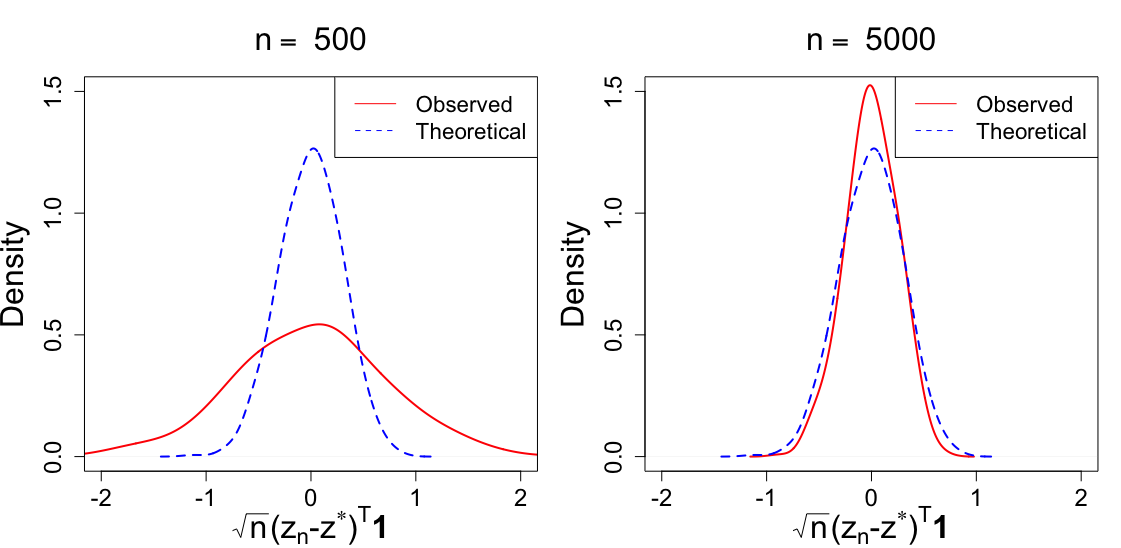}}
    \subfigure[]{\label{fig:mark}\includegraphics[width=80mm,height=38mm]{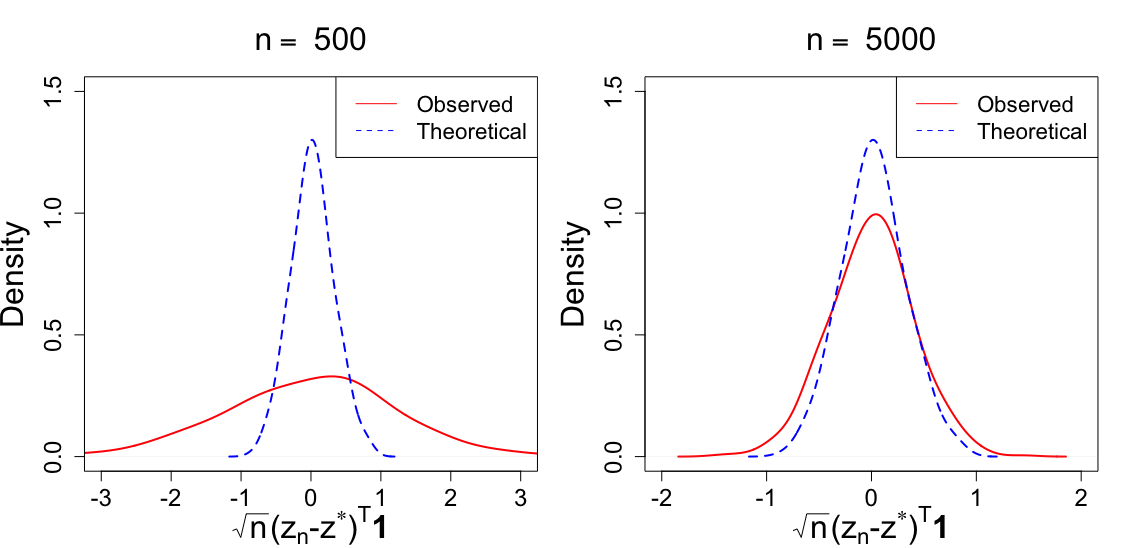}}
    \caption{Asymptotic distribution of $\bar{z}_n$: (a) martingale difference noise (b) state-dependent Markov noise.}\label{fig:state_dep_price_proj}
\end{figure}
\vspace{-0.1in}
To use a similar setup for both the noise settings, we simulate the martingale-difference noise by setting $A_1=B_1=A_2=B_2=\rho=0$ in the example in Section~\ref{sec:motvn}. Then \eqref{eq:evcharge} becomes a locally strongly-convex strongly-concave globally convex-concave game with martingale-difference noise. We adopt the experimental setup described in Section~\ref{sec:motvn} for the state-dependent Markov noise. We set $N=3$, $\gamma_{A,i}=1\mathbbm{1}(\norm{\theta}_2\leq 1)$, $\gamma_{B,i}=1\mathbbm{1}(\norm{\mu}_2\leq 1)$, $r_i=0.3$, $\rho=0.4$, $A_1=B_2=\text{diag}\{-0.3,-0.3,-0.3\}$, $A_2=B_1=\text{diag}\{0.3,0.3,0.3\}$, $D_{A}\sim N([0.1,0.1,0.1]^\top,I_3)$, $D_{B}\sim N([0,0,0]^\top,I_3)$ \cite{wood2022stochastic}. The noise $(a_k,b_k)$ is clearly a state-dependent Markov noise. Since the dynamics is driven by Gaussian noise, if we let $\mathcal{V}(a,b)=\norm{a}_2+\norm{b}_2$, then $\expec{V(a,b)^{\alpha_0}}{}<\infty$ for any $\alpha_0>2$. Since $0<\rho<1$, when $\theta$ and $\mu$ are fixed, the chain is geometrically ergodic \cite{grunwald2000theory}, thus satisfying condition 1 of Assumption~\ref{as:noisemar}. When $(\theta,\mu)$ belongs to a compact set, since the gradient is a linear function of $(a,b)$, $H(\theta,\mu,a,b)\lesssim \mathcal{V}(a,b)$. Since the transition kernel is linear in $\theta$ and $\mu$, and $A_1,B_1$, $A_2,B_2$ are bounded, condition 3 of Assumption~\ref{as:noisemar} holds. For the martingale difference and state-dependent Markov noise settings, the averaged iterates converge to $([0.2,0.2,0.2]^\top,[0.15,0.15,0.15]^\top)$, and $([0.22,0.22,0.22]^\top,[0.16,0.16,0.16]^\top)$ respectively, which agree with the theoretical values for $(\theta^*,\mu^*)$. Figure~\ref{fig:state_dep_price_proj} shows the observed and theoretical distributions of $\sqrt{n}\boldsymbol{1}^\top \left((\bar{\theta}_n,\bar{\mu}_n)-(\theta^*,\mu^*)\right)$, where $\boldsymbol{1}$ is a $2N$ dimensional vector with all elements equal to 1, after 500 and 5000 iterations for the martingale difference (Fig.~\ref{fig:martdiff}) and state-dependent Markov noise (Fig.~\ref{fig:mark}). The observed distributions converge to the theoretical distributions over time. 
\vspace{-0.1in}

\section{Conclusion and Discussion}
\vspace{-0.05in}
In this work, we show that for a convex-concave (locally strongly-convex strongly-concave) objective, the averaged iterates of SEG method converges almost surely to the saddle point (Theorem~\ref{th:asconv}), follows a CLT and achieves optimal asymptotic covariance (Theorem~\ref{th:martdiffclt}) for martingale-difference noise. In the state-dependent Markov noise setting, to ensure the stability of the algorithm, we propose a variant of SEG method with varying truncation sets. We show that the number of required truncations is finite almost surely. Then we establish almost sure convergence and CLT results in this setting (Theorem~\ref{th:asconvmar} and ~\ref{th:marclt}). Our work is a step forward toward constructing a confidence interval for the algorithmic estimator of saddle-point with potential applications in inference for multitask strategic classification, robust risk minimization, and relative cost maximization in competitive markets. An interesting future direction is obtaining non-asymptotic rates of convergence to the normal distribution with explicit dimension dependence, similar to \cite{shao2022berry}. Online covariance estimation to construct a confidence interval for $(\t^*,\m^*)$ is a challenging and intriguing problem as well. 

\bibliographystyle{plainnat}
\bibliography{main}

\begin{thebibliography}{83}
\providecommand{\natexlab}[1]{#1}
\providecommand{\url}[1]{\texttt{#1}}
\expandafter\ifx\csname urlstyle\endcsname\relax
  \providecommand{\doi}[1]{doi: #1}\else
  \providecommand{\doi}{doi: \begingroup \urlstyle{rm}\Url}\fi

\bibitem[Abernethy et~al.(2021)Abernethy, Lai, and Wibisono]{abernethy2021last}
Jacob Abernethy, Kevin~A Lai, and Andre Wibisono.
\newblock Last-iterate convergence rates for min-max optimization: Convergence
  of hamiltonian gradient descent and consensus optimization.
\newblock In \emph{Algorithmic Learning Theory}, pages 3--47. PMLR, 2021.

\bibitem[Adolphs(2018)]{adolphs2018non}
Leonard Adolphs.
\newblock Non convex-concave saddle point optimization.
\newblock Master's thesis, ETH Zurich, 2018.

\bibitem[Andrieu et~al.(2005)Andrieu, Moulines, and
  Priouret]{andrieu2005stability}
Christophe Andrieu, {\'E}ric Moulines, and Pierre Priouret.
\newblock Stability of stochastic approximation under verifiable conditions.
\newblock \emph{SIAM Journal on control and optimization}, 44\penalty0
  (1):\penalty0 283--312, 2005.

\bibitem[Antonakopoulos et~al.(2020)Antonakopoulos, Belmega, and
  Mertikopoulos]{antonakopoulos2020adaptive}
Kimon Antonakopoulos, E~Veronica Belmega, and Panayotis Mertikopoulos.
\newblock Adaptive extra-gradient methods for min-max optimization and games.
\newblock \emph{arXiv preprint arXiv:2010.12100}, 2020.

\bibitem[Asi and Duchi(2019)]{asi2019stochastic}
Hilal Asi and John~C Duchi.
\newblock Stochastic (approximate) proximal point methods: Convergence,
  optimality, and adaptivity.
\newblock \emph{SIAM Journal on Optimization}, 29\penalty0 (3):\penalty0
  2257--2290, 2019.

\bibitem[Bartlett(1992)]{bartlett1992learning}
Peter~L Bartlett.
\newblock Learning with a slowly changing distribution.
\newblock In \emph{Proceedings of the fifth annual workshop on Computational
  learning theory}, pages 243--252, 1992.

\bibitem[Benveniste et~al.(2012)Benveniste, M{\'e}tivier, and
  Priouret]{benveniste2012adaptive}
Albert Benveniste, Michel M{\'e}tivier, and Pierre Priouret.
\newblock \emph{Adaptive algorithms and stochastic approximations}, volume~22.
\newblock Springer Science \& Business Media, 2012.

\bibitem[Beznosikov et~al.(2021)Beznosikov, Samokhin, and
  Gasnikov]{beznosikov2021distributed}
Aleksandr Beznosikov, Valentin Samokhin, and Alexander Gasnikov.
\newblock Distributed saddle-point problems: Lower bounds, optimal algorithms
  and federated gans.
\newblock \emph{arXiv preprint arXiv:2010.13112}, 2021.

\bibitem[Beznosikov et~al.(2023)Beznosikov, Gorbunov, Berard, and
  Loizou]{beznosikov2023stochastic}
Aleksandr Beznosikov, Eduard Gorbunov, Hugo Berard, and Nicolas Loizou.
\newblock Stochastic gradient descent-ascent: Unified theory and new efficient
  methods.
\newblock In \emph{International Conference on Artificial Intelligence and
  Statistics}, pages 172--235. PMLR, 2023.

\bibitem[Bot et~al.(2019)Bot, Mertikopoulos, Staudigl, and
  Vuong]{bot2019forward}
Radu~Ioan Bot, Panayotis Mertikopoulos, Mathias Staudigl, and Phan~Tu Vuong.
\newblock Forward-backward-forward methods with variance reduction for
  stochastic variational inequalities.
\newblock \emph{arXiv preprint arXiv:1902.03355}, 2019.

\bibitem[Cai et~al.(2015)Cai, Daskalakis, and Papadimitriou]{cai2015optimum}
Yang Cai, Constantinos Daskalakis, and Christos Papadimitriou.
\newblock Optimum statistical estimation with strategic data sources.
\newblock In \emph{Conference on Learning Theory}, pages 280--296. PMLR, 2015.

\bibitem[Chen(2006)]{chen2006stochastic}
Han-Fu Chen.
\newblock \emph{Stochastic approximation and its applications}, volume~64.
\newblock Springer Science \& Business Media, 2006.

\bibitem[Clarke(1980)]{clarke1980probability}
LE~Clarke.
\newblock Probability and measure, by patrick billingsley. pp 515.{\pounds}
  15{\textperiodcentered} 20. 1979. sbn 0 471 03173 9 (wiley).
\newblock \emph{The Mathematical Gazette}, 64\penalty0 (430):\penalty0
  293--294, 1980.

\bibitem[Daskalakis et~al.(2017)Daskalakis, Ilyas, Syrgkanis, and
  Zeng]{daskalakis2017training}
Constantinos Daskalakis, Andrew Ilyas, Vasilis Syrgkanis, and Haoyang Zeng.
\newblock Training gans with optimism.
\newblock \emph{arXiv preprint arXiv:1711.00141}, 2017.

\bibitem[Diakonikolas and Orecchia(2017)]{diakonikolas2017accelerated}
Jelena Diakonikolas and Lorenzo Orecchia.
\newblock Accelerated extra-gradient descent: A novel accelerated first-order
  method.
\newblock \emph{arXiv preprint arXiv:1706.04680}, 2017.

\bibitem[Diakonikolas and Wang(2022)]{diakonikolas2022potential}
Jelena Diakonikolas and Puqian Wang.
\newblock Potential function-based framework for minimizing gradients in convex
  and min-max optimization.
\newblock \emph{SIAM Journal on Optimization}, 32\penalty0 (3):\penalty0
  1668--1697, 2022.

\bibitem[Diakonikolas et~al.(2021)Diakonikolas, Daskalakis, and
  Jordan]{diakonikolas2021efficient}
Jelena Diakonikolas, Constantinos Daskalakis, and Michael~I Jordan.
\newblock Efficient methods for structured nonconvex-nonconcave min-max
  optimization.
\newblock In \emph{International Conference on Artificial Intelligence and
  Statistics}, pages 2746--2754. PMLR, 2021.

\bibitem[Douc et~al.(2018)Douc, Moulines, Priouret, and
  Soulier]{douc2018markov}
Randal Douc, Eric Moulines, Pierre Priouret, and Philippe Soulier.
\newblock \emph{Markov chains}.
\newblock Springer, 2018.

\bibitem[Duchi et~al.(2012)Duchi, Agarwal, Johansson, and
  Jordan]{duchi2012ergodic}
John~C Duchi, Alekh Agarwal, Mikael Johansson, and Michael~I Jordan.
\newblock Ergodic mirror descent.
\newblock \emph{SIAM Journal on Optimization}, 22\penalty0 (4):\penalty0
  1549--1578, 2012.

\bibitem[Fallah et~al.(2020)Fallah, Ozdaglar, and Pattathil]{fallah2020optimal}
Alireza Fallah, Asuman Ozdaglar, and Sarath Pattathil.
\newblock An optimal multistage stochastic gradient method for minimax
  problems.
\newblock In \emph{2020 59th IEEE Conference on Decision and Control (CDC)},
  pages 3573--3579. IEEE, 2020.

\bibitem[Fort(2015)]{fort2015central}
Gersende Fort.
\newblock Central limit theorems for stochastic approximation with controlled
  markov chain dynamics.
\newblock \emph{ESAIM: Probability and Statistics}, 19:\penalty0 60--80, 2015.

\bibitem[Gidel et~al.(2018)Gidel, Berard, Vignoud, Vincent, and
  Lacoste-Julien]{gidel2018variational}
Gauthier Gidel, Hugo Berard, Ga{\"e}tan Vignoud, Pascal Vincent, and Simon
  Lacoste-Julien.
\newblock A variational inequality perspective on generative adversarial
  networks.
\newblock \emph{arXiv preprint arXiv:1802.10551}, 2018.

\bibitem[Gidel et~al.(2019)Gidel, Hemmat, Pezeshki, Le~Priol, Huang,
  Lacoste-Julien, and Mitliagkas]{gidel2019negative}
Gauthier Gidel, Reyhane~Askari Hemmat, Mohammad Pezeshki, R{\'e}mi Le~Priol,
  Gabriel Huang, Simon Lacoste-Julien, and Ioannis Mitliagkas.
\newblock Negative momentum for improved game dynamics.
\newblock In \emph{The 22nd International Conference on Artificial Intelligence
  and Statistics}, pages 1802--1811. PMLR, 2019.

\bibitem[Goldberg et~al.(2013)Goldberg, Song, and
  Kosorok]{goldberg2013adaptive}
Yair Goldberg, Rui Song, and Michael~R Kosorok.
\newblock Adaptive q-learning.
\newblock In \emph{From Probability to Statistics and Back: High-Dimensional
  Models and Processes--A Festschrift in Honor of Jon A. Wellner}, pages
  150--162. Institute of Mathematical Statistics, 2013.

\bibitem[Golowich et~al.(2020)Golowich, Pattathil, Daskalakis, and
  Ozdaglar]{golowich2020last}
Noah Golowich, Sarath Pattathil, Constantinos Daskalakis, and Asuman Ozdaglar.
\newblock Last iterate is slower than averaged iterate in smooth convex-concave
  saddle point problems.
\newblock In \emph{Conference on Learning Theory}, pages 1758--1784. PMLR,
  2020.

\bibitem[Goodfellow et~al.(2014)Goodfellow, Pouget-Abadie, Mirza, Xu,
  Warde-Farley, Ozair, Courville, and Bengio]{goodfellow2014generative}
Ian Goodfellow, Jean Pouget-Abadie, Mehdi Mirza, Bing Xu, David Warde-Farley,
  Sherjil Ozair, Aaron Courville, and Yoshua Bengio.
\newblock Generative adversarial nets in advances in neural information
  processing systems (nips).
\newblock \emph{Curran Associates, Inc. Red Hook, NY, USA}, pages 2672--2680,
  2014.

\bibitem[Gorbunov et~al.(2022)Gorbunov, Berard, Gidel, and
  Loizou]{gorbunov2022stochastic}
Eduard Gorbunov, Hugo Berard, Gauthier Gidel, and Nicolas Loizou.
\newblock Stochastic extragradient: General analysis and improved rates.
\newblock In \emph{International Conference on Artificial Intelligence and
  Statistics}, pages 7865--7901. PMLR, 2022.

\bibitem[Grunwald et~al.(2000)Grunwald, Hyndman, Tedesco, and
  Tweedie]{grunwald2000theory}
Gary~K Grunwald, Rob~J Hyndman, Leanna Tedesco, and Richard~L Tweedie.
\newblock Theory \& methods: Non-gaussian conditional linear ar (1) models.
\newblock \emph{Australian \& New Zealand Journal of Statistics}, 42\penalty0
  (4):\penalty0 479--495, 2000.

\bibitem[Hardt et~al.(2016)Hardt, Megiddo, Papadimitriou, and
  Wootters]{hardt2016strategic}
Moritz Hardt, Nimrod Megiddo, Christos Papadimitriou, and Mary Wootters.
\newblock Strategic classification.
\newblock In \emph{Proceedings of the 2016 ACM conference on innovations in
  theoretical computer science}, pages 111--122, 2016.

\bibitem[Holcblat and Sowell(2022)]{holcblat2022empirical}
Benjamin Holcblat and Fallaw Sowell.
\newblock The empirical saddlepoint estimator.
\newblock \emph{Electronic Journal of Statistics}, 16\penalty0 (1):\penalty0
  3672--3694, 2022.

\bibitem[Hsieh et~al.(2020)Hsieh, Iutzeler, Malick, and
  Mertikopoulos]{hsieh2020explore}
Yu-Guan Hsieh, Franck Iutzeler, J{\'e}r{\^o}me Malick, and Panayotis
  Mertikopoulos.
\newblock Explore aggressively, update conservatively: Stochastic extragradient
  methods with variable stepsize scaling.
\newblock \emph{Advances in Neural Information Processing Systems},
  33:\penalty0 16223--16234, 2020.

\bibitem[Huber(1973)]{Huber73annals}
P.~J. Huber.
\newblock Robust regression: Asymptotics, conjectures and {Monte Carlo}.
\newblock \emph{Ann. Stat.}, 1:\penalty0 799--821, 1973.

\bibitem[Iusem et~al.(2017)Iusem, Jofr{\'e}, Oliveira, and
  Thompson]{iusem2017extragradient}
Alfredo~N Iusem, Alejandro Jofr{\'e}, Roberto~Imbuzeiro Oliveira, and Philip
  Thompson.
\newblock Extragradient method with variance reduction for stochastic
  variational inequalities.
\newblock \emph{SIAM Journal on Optimization}, 27\penalty0 (2):\penalty0
  686--724, 2017.

\bibitem[Jin and Sidford(2020)]{jin2020efficiently}
Yujia Jin and Aaron Sidford.
\newblock Efficiently solving mdps with stochastic mirror descent.
\newblock In \emph{International Conference on Machine Learning}, pages
  4890--4900. PMLR, 2020.

\bibitem[Karimi et~al.(2019)Karimi, Miasojedow, Moulines, and
  Wai]{karimi2019non}
Belhal Karimi, Blazej Miasojedow, Eric Moulines, and Hoi-To Wai.
\newblock Non-asymptotic analysis of biased stochastic approximation scheme.
\newblock In \emph{Conference on Learning Theory}, pages 1944--1974. PMLR,
  2019.

\bibitem[Korpelevich(1976{\natexlab{a}})]{korpelevich1976extragradient}
Galina~M Korpelevich.
\newblock The extragradient method for finding saddle points and other
  problems.
\newblock \emph{Matecon}, 12:\penalty0 747--756, 1976{\natexlab{a}}.

\bibitem[Korpelevich(1976{\natexlab{b}})]{korpelevichextragradient}
GM~Korpelevich.
\newblock An extragradient method for finding saddle points and for other
  problems, ekonomika i matematicheskie metody, 12 (1976), 747--756.
\newblock \emph{Search in}, 1976{\natexlab{b}}.

\bibitem[Kushner and Clark(2012)]{kushner2012stochastic}
Harold~Joseph Kushner and Dean~S Clark.
\newblock \emph{Stochastic approximation methods for constrained and
  unconstrained systems}, volume~26.
\newblock Springer Science \& Business Media, 2012.

\bibitem[Lei and Shanbhag(2020)]{lei2020variance}
Jinlong Lei and Uday~V Shanbhag.
\newblock Variance-reduced accelerated first-order methods: Central limit
  theorems and confidence statements.
\newblock \emph{arXiv preprint arXiv:2006.07769}, 2020.

\bibitem[Levanon and Rosenfeld(2021)]{levanon2021strategic}
Sagi Levanon and Nir Rosenfeld.
\newblock Strategic classification made practical.
\newblock In \emph{International Conference on Machine Learning}, pages
  6243--6253. PMLR, 2021.

\bibitem[Li et~al.(2022{\natexlab{a}})Li, Yu, Loizou, Gidel, Ma, Le~Roux, and
  Jordan]{li2022convergence}
Chris~Junchi Li, Yaodong Yu, Nicolas Loizou, Gauthier Gidel, Yi~Ma, Nicolas
  Le~Roux, and Michael Jordan.
\newblock On the convergence of stochastic extragradient for bilinear games
  using restarted iteration averaging.
\newblock In \emph{International Conference on Artificial Intelligence and
  Statistics}, pages 9793--9826. PMLR, 2022{\natexlab{a}}.

\bibitem[Li and Wai(2022)]{li2022state}
Qiang Li and Hoi-To Wai.
\newblock State dependent performative prediction with stochastic
  approximation.
\newblock In \emph{International Conference on Artificial Intelligence and
  Statistics}, pages 3164--3186. PMLR, 2022.

\bibitem[Li et~al.(2022{\natexlab{b}})Li, Xiao, and Yang]{li2022revisiting}
Tiejun Li, Tiannan Xiao, and Guoguo Yang.
\newblock Revisiting the central limit theorems for the sgd-type methods.
\newblock \emph{arXiv preprint arXiv:2207.11755}, 2022{\natexlab{b}}.

\bibitem[Liang(2010)]{liang2010trajectory}
Faming Liang.
\newblock Trajectory averaging for stochastic approximation mcmc algorithms1.
\newblock \emph{The Annals of Statistics}, 38\penalty0 (5):\penalty0
  2823--2856, 2010.

\bibitem[Liang and Stokes(2019)]{liang2019interaction}
Tengyuan Liang and James Stokes.
\newblock Interaction matters: A note on non-asymptotic local convergence of
  generative adversarial networks.
\newblock In \emph{The 22nd International Conference on Artificial Intelligence
  and Statistics}, pages 907--915. PMLR, 2019.

\bibitem[Liu and Zhang(2022)]{liu2022confidence}
Yongchao Liu and Jin Zhang.
\newblock Confidence regions of stochastic variational inequalities: error
  bound approach.
\newblock \emph{Optimization}, 71\penalty0 (7):\penalty0 2157--2184, 2022.

\bibitem[Loizou et~al.(2021)Loizou, Berard, Gidel, Mitliagkas, and
  Lacoste-Julien]{loizou2021stochastic}
Nicolas Loizou, Hugo Berard, Gauthier Gidel, Ioannis Mitliagkas, and Simon
  Lacoste-Julien.
\newblock Stochastic gradient descent-ascent and consensus optimization for
  smooth games: Convergence analysis under expected co-coercivity.
\newblock \emph{Advances in Neural Information Processing Systems},
  34:\penalty0 19095--19108, 2021.

\bibitem[Mendler-D{\"u}nner et~al.(2020)Mendler-D{\"u}nner, Perdomo, Zrnic, and
  Hardt]{mendler2020stochastic}
Celestine Mendler-D{\"u}nner, Juan Perdomo, Tijana Zrnic, and Moritz Hardt.
\newblock Stochastic optimization for performative prediction.
\newblock \emph{Advances in Neural Information Processing Systems},
  33:\penalty0 4929--4939, 2020.

\bibitem[Mertens and Neyman(1981)]{mertens1981stochastic}
J~F Mertens and Abraham Neyman.
\newblock Stochastic games.
\newblock \emph{International Journal of Game Theory}, 10:\penalty0 53--66,
  1981.

\bibitem[Meyn and Tweedie(2012)]{meyn2012markov}
Sean~P Meyn and Richard~L Tweedie.
\newblock \emph{Markov chains and stochastic stability}.
\newblock Springer Science \& Business Media, 2012.

\bibitem[Minsker(2020)]{minsker2020asymptotic}
Stanislav Minsker.
\newblock Asymptotic normality of robust risk minimizers.
\newblock \emph{arXiv preprint arXiv:2004.02328}, 2020.

\bibitem[Mishchenko et~al.(2020)Mishchenko, Kovalev, Shulgin, Richt{\'a}rik,
  and Malitsky]{mishchenko2020revisiting}
Konstantin Mishchenko, Dmitry Kovalev, Egor Shulgin, Peter Richt{\'a}rik, and
  Yura Malitsky.
\newblock Revisiting stochastic extragradient.
\newblock In \emph{International Conference on Artificial Intelligence and
  Statistics}, pages 4573--4582. PMLR, 2020.

\bibitem[Mokhtari et~al.(2020)Mokhtari, Ozdaglar, and
  Pattathil]{mokhtari2020unified}
Aryan Mokhtari, Asuman Ozdaglar, and Sarath Pattathil.
\newblock A unified analysis of extra-gradient and optimistic gradient methods
  for saddle point problems: Proximal point approach.
\newblock In \emph{International Conference on Artificial Intelligence and
  Statistics}, pages 1497--1507. PMLR, 2020.

\bibitem[Mou et~al.(2020)Mou, Li, Wainwright, Bartlett, and
  Jordan]{mou2020linear}
Wenlong Mou, Chris~Junchi Li, Martin~J Wainwright, Peter~L Bartlett, and
  Michael~I Jordan.
\newblock On linear stochastic approximation: Fine-grained polyak-ruppert and
  non-asymptotic concentration.
\newblock In \emph{Conference on Learning Theory}, pages 2947--2997. PMLR,
  2020.

\bibitem[Mukherjee and Chakraborty(2020)]{mukherjee2020decentralized}
Soham Mukherjee and Mrityunjoy Chakraborty.
\newblock A decentralized algorithm for large scale min-max problems.
\newblock In \emph{2020 59th IEEE Conference on Decision and Control (CDC)},
  pages 2967--2972. IEEE, 2020.

\bibitem[Nagaraj et~al.(2020)Nagaraj, Wu, Bresler, Jain, and
  Netrapalli]{nagaraj2020least}
Dheeraj Nagaraj, Xian Wu, Guy Bresler, Prateek Jain, and Praneeth Netrapalli.
\newblock Least squares regression with {M}arkovian data: Fundamental limits
  and algorithms.
\newblock \emph{Advances in neural information processing systems},
  33:\penalty0 16666--16676, 2020.

\bibitem[Nemirovski(2004)]{nemirovski2004prox}
Arkadi Nemirovski.
\newblock Prox-method with rate of convergence o (1/t) for variational
  inequalities with lipschitz continuous monotone operators and smooth
  convex-concave saddle point problems.
\newblock \emph{SIAM Journal on Optimization}, 15\penalty0 (1):\penalty0
  229--251, 2004.

\bibitem[Pelletier(2000)]{pelletier2000asymptotic}
Mariane Pelletier.
\newblock Asymptotic almost sure efficiency of averaged stochastic algorithms.
\newblock \emph{SIAM Journal on Control and Optimization}, 39\penalty0
  (1):\penalty0 49--72, 2000.

\bibitem[Polyak and Juditsky(1992)]{polyak1992acceleration}
Boris~T Polyak and Anatoli~B Juditsky.
\newblock Acceleration of stochastic approximation by averaging.
\newblock \emph{SIAM journal on control and optimization}, 30\penalty0
  (4):\penalty0 838--855, 1992.

\bibitem[Portnoy(1985)]{Portnoy85annals}
S.~Portnoy.
\newblock Asymptotic behavior of {M}-estimators of $p$ regression parameters
  when $p^2/n$ is large; ii. {N}ormal approximation.
\newblock \emph{Ann. Stat.}, 13:\penalty0 1403--1417, 1985.

\bibitem[Qu and Wierman(2020)]{qu2020finite}
Guannan Qu and Adam Wierman.
\newblock Finite-time analysis of asynchronous stochastic approximation and $ q
  $-learning.
\newblock In \emph{Conference on Learning Theory}, pages 3185--3205. PMLR,
  2020.

\bibitem[Rakhlin and Sridharan(2013)]{rakhlin2013online}
Alexander Rakhlin and Karthik Sridharan.
\newblock Online learning with predictable sequences.
\newblock In \emph{Conference on Learning Theory}, pages 993--1019. PMLR, 2013.

\bibitem[Roy and Balasubramanian(2023)]{roy2023online}
Abhishek Roy and Krishnakumar Balasubramanian.
\newblock Online covariance estimation for stochastic gradient descent under
  markovian sampling.
\newblock \emph{arXiv preprint arXiv:2308.01481}, 2023.

\bibitem[Roy et~al.(2022)Roy, Balasubramanian, and Ghadimi]{royconstrained}
Abhishek Roy, Krishna Balasubramanian, and Saeed Ghadimi.
\newblock Constrained stochastic nonconvex optimization with state-dependent
  {M}arkov data.
\newblock In \emph{Advances in Neural Information Processing Systems}, 2022.

\bibitem[Ruppert(1988)]{ruppert1988efficient}
David Ruppert.
\newblock Efficient estimations from a slowly convergent robbins-monro process.
\newblock Technical report, Cornell University Operations Research and
  Industrial Engineering, 1988.

\bibitem[Shao and Zhang(2022)]{shao2022berry}
Qi-Man Shao and Zhuo-Song Zhang.
\newblock Berry--esseen bounds for multivariate nonlinear statistics with
  applications to m-estimators and stochastic gradient descent algorithms.
\newblock \emph{Bernoulli}, 28\penalty0 (3):\penalty0 1548--1576, 2022.

\bibitem[Shapiro and Kleywegt(2002)]{shapiro2002minimax}
Alexander Shapiro and Anton Kleywegt.
\newblock Minimax analysis of stochastic problems.
\newblock \emph{Optimization Methods and Software}, 17\penalty0 (3):\penalty0
  523--542, 2002.

\bibitem[Su et~al.(2014)Su, Boyd, and Candes]{su2014differential}
Weijie Su, Stephen Boyd, and Emmanuel Candes.
\newblock A differential equation for modeling nesterov’s accelerated
  gradient method: theory and insights.
\newblock \emph{Advances in neural information processing systems}, 27, 2014.

\bibitem[Sur and Cand\'{e}s(2019)]{Sur19pnas}
Pragya Sur and Emmanuel~J. Cand\'{e}s.
\newblock A modern maximum-likelihood theory for high-dimensional logistic
  regression.
\newblock \emph{Proceedings of the National Academy of Sciences (PNAS)},
  116\penalty0 (29):\penalty0 14516--14525, 2019.

\bibitem[Tiapkin et~al.(2022)Tiapkin, Gasnikov, and
  Dvurechensky]{tiapkin2022stochastic}
Daniil Tiapkin, Alexander Gasnikov, and Pavel Dvurechensky.
\newblock Stochastic saddle-point optimization for the wasserstein barycenter
  problem.
\newblock \emph{Optimization Letters}, 16\penalty0 (7):\penalty0 2145--2175,
  2022.

\bibitem[Toulis et~al.(2021)Toulis, Horel, and Airoldi]{toulis2021proximal}
Panos Toulis, Thibaut Horel, and Edoardo~M Airoldi.
\newblock The proximal robbins--monro method.
\newblock \emph{Journal of the Royal Statistical Society Series B: Statistical
  Methodology}, 83\penalty0 (1):\penalty0 188--212, 2021.

\bibitem[Tseng(1995)]{tseng1995linear}
Paul Tseng.
\newblock On linear convergence of iterative methods for the variational
  inequality problem.
\newblock \emph{Journal of Computational and Applied Mathematics}, 60\penalty0
  (1-2):\penalty0 237--252, 1995.

\bibitem[Vasile(2014)]{vasile2014solution}
Massimiliano Vasile.
\newblock On the solution of min-max problems in robust optimization.
\newblock In \emph{The EVOLVE 2014 International Conference, A Bridge between
  Probability, Set Oriented Numerics, and Evolutionary Computing}, 2014.

\bibitem[Wang et~al.(2022)Wang, Lei, Ying, and Zhou]{wang2022stability}
Puyu Wang, Yunwen Lei, Yiming Ying, and Ding-Xuan Zhou.
\newblock Stability and generalization for {M}arkov chain stochastic gradient
  methods.
\newblock \emph{arXiv preprint arXiv:2209.08005}, 2022.

\bibitem[Wood and Dall'Anese(2022)]{wood2022stochastic}
Killian Wood and Emiliano Dall'Anese.
\newblock Stochastic saddle point problems with decision-dependent
  distributions.
\newblock \emph{arXiv preprint arXiv:2201.02313}, 2022.

\bibitem[Wu et~al.(2020)Wu, Zhang, Xu, and Gu]{wu2020finite}
Yue~Frank Wu, Weitong Zhang, Pan Xu, and Quanquan Gu.
\newblock A finite-time analysis of two time-scale actor-critic methods.
\newblock \emph{Advances in Neural Information Processing Systems},
  33:\penalty0 17617--17628, 2020.

\bibitem[Xu et~al.(2023)Xu, Zhang, Xu, and Lan]{xu2023unified}
Zi~Xu, Huiling Zhang, Yang Xu, and Guanghui Lan.
\newblock A unified single-loop alternating gradient projection algorithm for
  nonconvex--concave and convex--nonconcave minimax problems.
\newblock \emph{Mathematical Programming}, pages 1--72, 2023.

\bibitem[Yan et~al.(2022)Yan, Lin, Li, Ai, Shi, Zhang, and
  Gejirifu]{yan2022many}
Qingyou Yan, Hongyu Lin, Jinmeng Li, Xingbei Ai, Mengshu Shi, Meijuan Zhang,
  and De~Gejirifu.
\newblock Many-objective charging optimization for electric vehicles
  considering demand response and multi-uncertainties based on {M}arkov chain
  and information gap decision theory.
\newblock \emph{Sustainable Cities and Society}, 78:\penalty0 103652, 2022.

\bibitem[Yan and Liu(2022)]{yan2022stochastic}
Wuwenqing Yan and Yongchao Liu.
\newblock Stochastic approximation based confidence regions for stochastic
  variational inequalities.
\newblock \emph{arXiv preprint arXiv:2203.09933}, 2022.

\bibitem[Yang et~al.(2020)Yang, Kiyavash, and He]{yang2020global}
Junchi Yang, Negar Kiyavash, and Niao He.
\newblock Global convergence and variance reduction for a class of
  nonconvex-nonconcave minimax problems.
\newblock \emph{Advances in Neural Information Processing Systems},
  33:\penalty0 1153--1165, 2020.

\bibitem[Zhang et~al.(2021)Zhang, Janson, and Murphy]{zhang2021statistical}
Kelly Zhang, Lucas Janson, and Susan Murphy.
\newblock Statistical inference with m-estimators on adaptively collected data.
\newblock \emph{Advances in Neural Information Processing Systems}, 34, 2021.

\bibitem[Zhang and He(2020)]{zhang2020spatiotemporal}
Li~Zhang and Zheng He.
\newblock Spatiotemporal distribution model of charging demand for electric
  vehicle charging stations based on {M}arkov chain.
\newblock In \emph{Journal of Physics: Conference Series}, volume 1601, page
  022045. IOP Publishing, 2020.

\bibitem[Zhu et~al.(2021)Zhu, Chen, and Wu]{zhu2021online}
Wanrong Zhu, Xi~Chen, and Wei~Biao Wu.
\newblock Online covariance matrix estimation in stochastic gradient descent.
\newblock \emph{Journal of the American Statistical Association}, pages 1--12,
  2021.

\end{thebibliography}
\newpage
\appendix
\onecolumn
\section{Appendix}
\subsection{Proof for Section~\ref{sec:martdiff}}
\subsubsection{Proof for Theorem~\ref{th:asconv}}
\begin{proof}[Proof of Theorem~\ref{th:asconv}]\label{pf:asconv}
\begin{align*}
    &\norm{z_{k+1}-z^*}_2^2\\
    =&\norm{z_{k}-\eta_{k+1}H(\zh,w_{k+1})-z^*}_2^2\\
    =&\norm{z_{k}-z^*-\eta_{k+1}H(z_\kh)+\eta_{k+1}\xi(z_\kh,w_{k+1})}_2^2\\
    =&\norm{z_k-z^*}_2^2+\eta_{k+1}^2\norm{H(z_\kh,w_{k+1})}_2^2-2\eta_{k+1}(z_k-z^*)^\top H(z_\kh)
    +2\eta_{k+1}(z_k-z^*)^\top \xi(z_\kh,w_{k+1}).\numberthis\label{eq:zk1zkstardeomposition}
\end{align*}
Using Assumption~\ref{as:noise}, we have,
\begin{align*}
    &\expec{(z_k-z^*)^\top \xi(z_\kh,w_{k+1})|\cF_k}{}\\
    =&\expec{(z_k-z^*)^\top \xi(z_k,w_{k+1})|\cF_k}{}+\expec{(z_k-z^*)^\top (\xi(z_\kh,w_{k+1})-\xi(z_k,w_{k+1}))|\cF_k}{}\numberthis\label{eq:interaction}\\
    \overset{(a)}{=}&\expec{(z_k-z^*)^\top (\xi(z_\kh,w_{k+1})-\xi(z_k,w_{k+1}))|\cF_k}{}\\
    \leq &\expec{\norm{z_k-z^*}_2\norm{\xi(z_\kh,w_{k+1})-\xi(z_k,w_{k+1})}_2|\cF_k}{}\\
   \overset{(b)} {\leq} &L_G\expec{\norm{z_k-z^*}_2\norm{z_\kh-z_k}_2|\cF_k}{}\\
    \leq &\eta_{k+1}L_G\norm{z_k-z^*}_2\expec{\norm{H(z_k,w_{k+1})}_2|\cF_k}{}\\
    \overset{(c)}{\leq} &C\eta_{k+1}L_G\norm{z_k-z^*}_2\left(1+\norm{z_k-z^*}_2\right)\\
    \overset{(d)}{\leq} &C\eta_{k+1}L_G\norm{z_k-z^*}_2^2+C\eta_{k+1}L_G\left(1+\norm{z_k-z^*}_2^2\right)\\
    \leq &C\eta_{k+1}L_G\norm{z_k-z^*}_2^2+C\eta_{k+1}L_G.\numberthis\label{eq:term6bound}
\end{align*}
(a) follows from Assumption~\ref{as:noise}. (b) follows from Assumption~\ref{as:lipgrad}. (c) follows from \eqref{eq:noisegradvar}. (d) follows using Young's inequality. By Assumption~\ref{as:strongcon}, since $f(\theta,\cdot)$ is convex for each $\mu$, and $f(\cdot,\mu)$ is concave for each $\theta$, for any $(\theta,\mu)$ we have,
\begin{align*}
    \nabla_\theta f(\theta,\mu)^\top(\theta-\theta^*)\geq f(\theta,\mu)-f(\theta^*,\mu), \qquad -\nabla_\mu f(\theta,\mu)^\top(\mu-\mu^*)\geq -f(\theta,\mu)+f(\theta,\mu^*).
\end{align*}
Combining the above inequalities one gets,
\begin{align*}
    H(z_k)^\top (z_k-z^*)\geq \subop(z_k), \numberthis\label{eq:subopcc}
\end{align*}
where $\subop(z_k)\coloneqq f(\theta_k,\mu^*)-f(\theta^*,\mu_k)$ represents a suboptimality-measure for the saddle-point optimization. 
\begin{align*}
    &-2\eta_{k+1}(z_k-z^*)^\top H(z_\kh)\\
    =&-2\eta_{k+1}(z_k-z^*)^\top H(z_k)-2\eta_{k+1}(z_k-z^*)^\top (H(z_\kh)-H(z_k))\\
    \overset{(a)}{\leq} & -2\eta_{k+1}\subop(z_k)+\eta_{k+1}^2\norm{z_k-z^*}_2^2+\norm{H(z_\kh)-H(z_k)}_2^2\\
    \overset{(b)}{\leq} & -2\eta_{k+1}\subop(z_k)+\eta_{k+1}^2\norm{z_k-z^*}_2^2+L_G^2\norm{z_\kh-z_k}_2^2\\
    \overset{(c)}{=}&-2\eta_{k+1}\subop(z_k)+\eta_{k+1}^2\norm{z_k-z^*}_2^2+L_G^2\eta_{k+1}^2\norm{H(z_k,w_{k+1})}_2^2.\numberthis\label{eq:term4upperbound}
\end{align*}
Inequality (a) follows using \eqref{eq:subopcc}, and Young's inequality. (b) follows from Assumption~\ref{as:lipgrad}, and the fact $H(z^*)=0$. (c) follows from the update of Algorithm~\ref{alg:seg}, and Jensen's inequality. 
Using Jensen's inequality and Assumption~\ref{as:lipgrad}, we have the following bound on the second term in \eqref{eq:zk1zkstardeomposition}. 
\begin{align*}
    &\norm{H(z_\kh,w_{k+1})}_2^2\\
    \leq & 2\norm{H(z_\kh,w_{k+1})-H(z_k,w_{k+1})}_2^2+2\norm{H(z_k,w_{k+1})}_2^2\\
\leq & (2+2L_G^2\etak1^2)\etak1^2\norm{H(z_k,w_{k+1})}_2^2.\numberthis\label{eq:Hzk1bound}
\end{align*}
Combining \eqref{eq:zk1zkstardeomposition}, \eqref{eq:term6bound}, \eqref{eq:term4upperbound}, and \eqref{eq:Hzk1bound}, using \eqref{eq:noisegradvar}, and choosing $\etak1$ small enough, we get,
\begin{align*}
     \expec{\norm{z_{k+1}-z^*}_2^2|\cF_k}{}\leq
     \left(1+C'\etak1^2\right)\norm{z_{k}-z^*}_2^2-2\etak1\subop(z_k)+C'\eta_{k+1}^2.\numberthis\label{eq:recursionintermed}
\end{align*}
where $C'>0$ is a constant. 

Using Robbins-Siegmund theorem, we get, $\norm{z_{k+1}-z^*}_2^2$ converges almost surely to a finite random variable $z_\infty$, and $\sum_{k=1}^\infty\etak1\subop(z_k)<\infty$ almost surely. Now we will show that $z_\infty=0$. We will show this by contradiction. Let $\norm{z_k-z^*}_2>\omega>0$ almost surely. Then, by the local strong-convexity as described in Assumption~\ref{as:strongcon}, we have a constant $c_\omega>0$ such that $\subop(z_k)>c_\omega>0$ almost surely. Then, by our step-size choice (Assumption~\ref{as:stepsize}), almost surely,
\begin{align*}
    \sum_{k=1}^\infty\etak1\subop(z_k)>c_\omega\sum_{k=1}^\infty\etak1=\infty.
\end{align*}
But we established above that $\sum_{k=1}^\infty\etak1\subop(z_k)<\infty$ which is a contradiction. So $z_\infty=0$, i.e., 
$z_k\overset{a.s.}{\to}z^*$. Then, we have that for any $\epsilon>0$, there exists a $K\geq 1$ such that for all $k\geq K$, we have, $\norm{{z}_k-z^*}_2<\epsilon/2$. But,
\begin{align*}
    \norm{\bar{z}_k-z^*}_2\leq\frac{1}{k}\norm{\sum_{i=1}^{K}(z_i-z^*)}_2+\norm{\frac{1}{k}\sum_{i=K+1}^k(z_i-z^*)}_2<\frac{1}{k}\norm{\sum_{i=1}^{K}(z_i-z^*)}_2+\frac{\epsilon}{2}.
\end{align*}
Now we choose a sufficiently large $k=K_1\geq K$ such that $\norm{\sum_{i=1}^{K}(z_i-z^*)}_2/k\leq\epsilon/2$. Since $\epsilon$ is arbitrary, we have that $\bar{z}_k\overset{a.s.}{\to}z^*$.
\end{proof}
\subsubsection{Proof for Theorem~\ref{th:martdiffclt}}\label{sec:martdiffcltproof}
We first prove the result when $H(z_k)$ is linear in $z_k$, i.e., $H(z_k)=Qz_k$ for some matrix $Q$ such that $Re\left(\lambda_i(Q)\right)>0$ for all $i=1,2,\cdots,d_\theta+d_\mu$ where $Re(x)$ denote the real part of $x$. Let $\lambda_{min}(Q)$ be the minimum eigen value of $Q$, and $\gamma\coloneqq\min(\lambda_{min}(Q),1/(2\eta))$. Here the proof follows similar techniques to \citep{polyak1992acceleration} except that we use a different algorithm
Let us introduce the following notations which we will use through the analysis. 
\begin{align*}
    Y_i^j\coloneqq \prod_{k=i+1}^j(\id-\eta_k Q) \quad Y_i^i\coloneqq\id \quad \text{for all} \quad j>i, \ i=1,2,\cdots
\end{align*}
First we state the following Lemma and Proposition which we will need for the  proof of Theorem~\ref{th:martdiffclt}.
\begin{lemma}[\citep{zhu2021online}]\label{lm:Yijbound}
\begin{align*}
    \norm{Y_i^j}_2\leq \exp\left(-\frac{\gamma\eta}{1-a}\left((j+1)^a-(i+1)^a\right)\right).
\end{align*}
\end{lemma}
\begin{lemma}[\citep{zhu2021online}]\label{lm:sijbound}
Define $S_i^j\coloneqq\sum_{k={i+1}}^jY_i^k$ for $j>i$, and let $S_i^i\coloneqq 0$. Then,
\begin{align*}
    \norm{S_i^j}_2\leq \sum_{k={i+1}}^j\norm{Y_i^k}_2\lesssim (i+1)^a.
\end{align*}
\end{lemma}
\begin{lemma}\label{lm:noisedecomposition}
The following decomposition of the noise takes place.
\begin{align*}
    \xi(z_k,w_{k+1})= \xi(z^*,w_{k+1})+g(z_k,w_{k+1}),
\end{align*}
where,
\begin{enumerate}
    \item $\xi(z^*,w_{k+1})$ is a martingale-difference sequence.
    \item $
    \expec{\xi(z^*,w_{k+1})\xi(z^*,w_{k+1})^\top|\cF_{k}}{}\overset{P}{\to}\Sigma.
$
\item $\expec{\norm{g(z_k,w_{k+1})}_2^2|\cF_{k}}{}\lesssim \norm{z_k-z^*}_2^2$.
\end{enumerate}
\end{lemma}
The proof of Lemma~\ref{lm:noisedecomposition} is in Appendix~\ref{pf:noisedecomposition}. In the following Proposition we prove the result of Theorem~\ref{th:martdiffclt} for the linear case $H(z)=Qz$.
\begin{proposition}\label{pro:linclt}
Let $Q$ be such that $Re\left(\lambda_i(Q)\right)>0$, and Assumption~\ref{as:lipgrad},\ref{as:noise}, \ref{as:asymcovar}, and \ref{as:lindeberg} be true. Let $H(z_k)=Qz_k$. Then, for Algorithm~\ref{alg:seg}, choosing $\eta_k=\eta k^{-a}$ for sufficiently small $\eta$, and $1/2<a<1$, we have,
\begin{align*}
    \sqrt{k}(\bar{z}_k-z^*)\overset{d}{\to}N(0,Q^{-1}\Sigma Q^{-1}).
\end{align*}
\end{proposition}
The proof of Proposition~\ref{pro:linclt} is in Appendix~\ref{pf:linclt}. 
\begin{proof}[Proof of Theorem~\ref{th:martdiffclt}]\label{pf:martdiffclt}

Define $V_k\coloneqq z_k-z^*$.
\begin{align*}
    &V_{k+1}\\
    =&V_{k}-\eta_{k+1}H(\zh,w_{k+1})\\
    =&V_{k}-\eta_{k+1}(H(\zh,w_{k+1})-H(z_k,w_{k+1}))-\eta_{k+1}H(z_k)+\eta_{k+1}\xi(z_k,w_{k+1})\\
    =&(\id-\eta_{k+1}Q^*)V_k-\eta_{k+1}(H(\zh,w_{k+1})-H(z_k,w_{k+1}))-\eta_{k+1}(H(z_k)-Q^*V_k)+\eta_{k+1}\xi(z_k,w_{k+1})\\
    =&Y_0^{k+1}V_0-\sum_{j=1}^{k+1}\eta_{j}Y_j^{k+1}(H(z_{j-1/2},w_{j})-H(z_{j-1},w_{j}))-\sum_{j=1}^{k+1}\eta_{j}Y_j^{k+1}(H(z_{j-1})-Q^*V_{j-1})+\sum_{j=1}^{k+1}\eta_{j}Y_j^{k+1}\xi(z_{j-1},w_{j}). \\\numberthis\label{eq:HzklinearVkintermed}
\end{align*}
Then, 
\begin{align*}
    \sqrt{k}\bar{V}_{k}=&\frac{1}{\sqrt{k}}\sum_{j=1}^kY_0^kV_0-\underbrace{\frac{1}{\sqrt{k}}\sum_{i=1}^k\sum_{j=1}^{i}\eta_jY_j^i(H(z_{j-1/2},w_j)-H(z_{j-1},w_j))}_{T_2}+\frac{1}{\sqrt{k}}\sum_{i=1}^k\sum_{j=1}^{i}\eta_jY_j^i\xi(z_{j-1},w_{j})\\
    &-\underbrace{\frac{1}{\sqrt{k}}\sum_{i=1}^k\sum_{j=1}^{i}\eta_jY_j^i(H(z_{j-1})-Q^*V_{j-1})}_{T_4}.\numberthis\label{eq:HzklinearVkintermednonlin}
\end{align*}
Now if one can show that $\expec{\norm{T_4}_2}{}\to 0$, then Theorem~\ref{th:martdiffclt} follows from Proposition~\ref{pro:linclt}.

For any $\mathcal{D}>0$ such that $\{z|\norm{z-z^*}\leq \mc D\}\subset \mc Z$, define $\mathcal{E}_\mathcal{D}\coloneqq \inf_{k\geq 1}\{\norm{z_k-z^*}_2>\mathcal{D}\}$ where in $\mc Z$, $f$ is strongly-convex strongly-concave as defined in Assumption~\ref{as:strongcon}. Then, choosing $\eta_k\leq m/(24C)$, from \eqref{eq:recursionintermed}, we get
\begin{align*}
    &\expec{\norm{z_{k}-z^*}_2^2\mathbbm{1}(\mathcal{E}_{\mc D}>k)|\cF_{k-1}}{}\\
    \leq & \expec{\norm{z_{k}-z^*}_2^2\mathbbm{1}(\mathcal{E}_{\mc D}>k-1)|\cF_{k-1}}{}\\
    =&  \mathbbm{1}(\mathcal{E}_{\mc D}>k-1)\expec{\norm{z_{k}-z^*}_2^2|\cF_{k-1}}{}\\
    \leq & \mathbbm{1}(\mathcal{E}_{\mc D}>k-1)\left( \left(1+3C\eta_k^2\right)\norm{z_{k-1}-z^*}_2^2-\frac{m\eta_k}{4}\norm{z_{k-1}-z^*}_2^2+C(L_G+3)\eta_k^2\right)\\
    \leq &  \left(1-\frac{m\eta_k}{8}\right)\norm{z_{k-1}-z^*}_2^2\mathbbm{1}(\mathcal{E}_{\mc D}>k-1)+(L_G+3)C\eta_k^2
\end{align*}
Then,
\begin{align*}
    \expec{\norm{z_{k}-z^*}_2^2\mathbbm{1}(\mathcal{E}_{\mc D}>k)}{}\leq C_1\eta_k, \numberthis\label{eq:indicatorbound}
\end{align*}
for some constant $C_1>0$. 
By Assumption~\ref{as:strongcon} we have $Re\left(\lambda(Q^*)\right)>0$. This means that when the iterates are confined in $\mc Z$, $\expec{\norm{z_k-z^*}_2^2}{}$ decays at the rates of $\eta_k$.

Since $z_k\overset{a.s.}{\to}z^*$ (Theorem~\ref{th:asconv}), for every $\mc D>0$, we have a positive integer $K_{\mc D}$ such that for all $k\geq K_{\mc D}$, we have $\norm{z_k-z^*}_2\leq \mc D$. Intuitively, this implies that there is some finite $K_1\geq 1$, such that $z_k\in\mc Z$ for all $k\geq K_1$. Note that to prove Theorem~\ref{th:martdiffclt}, it is sufficient to establish the asymptotic normality for the tail of the sequence $\{\bar{z}_n\}$. So, for better exposition, from this point onwards, we will assume that the iterates are confined  in $\mc Z$, i.e., without loss generality, set $K_1=1$.\\
\textbf{Bound on $\expec{\norm{T_2}_2^2}{}$}

Using Assumption~\ref{as:lipgrad}, and Lemma~\ref{lm:sijbound}, we get,
\begin{align*}
\norm{T_2}_2^2
    =&\norm{\frac{1}{\sqrt{k}}\sum_{i=1}^k\sum_{j=1}^{i}\eta_jY_j^i(H(z_{j-1/2},w_j)-H(z_{j-1},w_j))}_2^2\\
    = &\norm{\frac{1}{\sqrt{k}}\sum_{j=1}^k\sum_{i=j}^{k}\eta_jY_j^i(H(z_{j-1/2},w_j)-H(z_{j-1},w_j))}_2^2\\
    \leq  &\frac{1}{k}\left(\sum_{j=1}^k\eta_j\norm{\sum_{i=j}^{k}Y_j^i}_2\norm{H(z_{j-1/2},w_j)-H(z_{j-1},w_j)}_2\right)^2\\
     \leq &\frac1k\left(\sum_{j=1}^kL_G\norm{z_{j-1/2}-z_{j-1}}_2\right)^2\\
     \leq &\frac1k\left(\sum_{j=1}^kL_G\eta_j\norm{H(z_{j-1},w_j)}_2\right)^2.
\end{align*}
Note that, using \eqref{eq:noisegradvar}, we have
\begin{align*}
    \expec{\norm{H(z_{j-1},w_{j})}_2^2}{}\leq C\left(1+\expec{\norm{z_{j-1}-z^*}_2^2}{}\right).
\end{align*}
Then, using Cauchy-Schwarz inequality,
\begin{align*}
    \expec{\norm{T_2}_2^2}{}\leq \frac1k\sum_{j=1}^k\sum_{i=1}^kCL_G^2\eta_j\eta_i\sqrt{\left(1+\expec{\norm{z_{j-1}-z^*}_2^2}{}\right)}\sqrt{\left(1+\expec{\norm{z_{i-1}-z^*}_2^2}{}\right)}.\numberthis\label{eq:t2boundintermed}
\end{align*}
Then, from \eqref{eq:t2boundintermed}, we have,
\begin{align*}
    \expec{\norm{T_2}_2^2}{}\leq \frac1k\sum_{i=1}^k\sum_{j=1}^kCL_G^2\eta_i\eta_j\lesssim \frac1k\left(\sum_{j=1}^kj^{-a}\right)^2\lesssim \frac1k\left(\int_{1}^kj^{-a}dj\right)^2\lesssim k^{1-2a}.\numberthis\label{eq:T2bound}
\end{align*}
\textbf{Bound on $\expec{\norm{T_4}_2}{}$}\\
By Assumption~\ref{as:strongcon}, 
\begin{align*}
    &\norm{T_4}_2\\
    =&\norm{\frac{1}{\sqrt{k}}\sum_{i=1}^k\sum_{j=1}^{i}\eta_jY_j^i(H(z_{j-1})-Q^*V_{j-1})}_2\\
    =&\norm{\frac{1}{\sqrt{k}}\sum_{j=1}^k\eta_j\left(\sum_{i=j}^{k}Y_j^i\right)(H(z_{j-1})-Q^*V_{j-1})}_2\\
    \leq & \frac{1}{\sqrt{k}}\sum_{j=1}^{k}\eta_j\norm{\left(\sum_{i=j}^{k}Y_j^i\right)}_2\norm{(H(z_{j-1})-Q^*V_{j-1})}_2\\
    \overset{(a)}{\leq}  & \frac{1}{\sqrt{k}}\sum_{j=1}^{k}\norm{(H(z_{j-1})-Q^*V_{j-1})}_2\\
    \leq & \frac{1}{\sqrt{k}}\sum_{j=1}^k\norm{z_{j-1}-z^*}_2^2\\
    \lesssim & \sum_{j=1}^\infty\frac{1}{\sqrt{j}}\norm{z_{j-1}-z^*}_2^2.\numberthis\label{eq:T4bound}
\end{align*}
Now if we can show that,
\begin{align*}
    \sum_{j=1}^\infty\frac{1}{\sqrt{j}}\norm{z_{j-1}-z^*}_2^2<\infty \numberthis\label{eq:finitesumofdiff}
\end{align*}
then using Kronecker's Lemma we can say $\expec{\norm{T_4}_2}{}\to 0$. 
From \eqref{eq:indicatorbound}, we have 
\begin{align*}
    \sum_{j=1}^\infty\frac{1}{\sqrt{j}}\expec{\norm{z_{j-1}-z^*}_2^2}{}\leq C_1\sum_{j=1}^\infty\frac{1}{\sqrt{j}}\eta_{j-1}<\infty.
\end{align*} 
Then using Proposition~\ref{pro:linclt}, we have, 
\begin{align*}
    \sqrt{k}(\bar{z}_k-z^*)\overset{d}{\to}N(0,{Q^*}^{-1}\Sigma {Q^*}^{-1}).
\end{align*}
\textbf{Optimality of covariance:} Now we show that the asymptotic covariance is indeed optimal. We have shown that the dynamics in the nonlinear case is asymptotically equivalent to the linear case when $Q=Q^*$. From \citep{pelletier2000asymptotic}, we know that, for dynamics of the form, 
\begin{align*}
    z_{k+1}=z_k-\eta_{k+1}A(H(z_k)+\varsigma_{k+1}),\numberthis\label{eq:genupdate}
\end{align*}
where $A$ is a matrix such that $(-AQ^*+\id/2)$ is stable, the optimal asymptotic covariance is given by ${Q^*}^{-1}\Sigma_\infty {Q^*}^{-1}$, where $\Sigma_\infty=\lim_{k\to\infty}\expec{\varsigma_k\varsigma_k^\top|\cF_k}{}$. This is achieved at $A={Q^*}^{-1}$. Analogous to \eqref{eq:genupdate}, $\Sigma_\infty$ for our setting is given by,
\begin{align*}
    &\lim_{k\to\infty}\expec{\left(\xi(z_k,w_{k+1})+\tilde{\xi}(z_k,w_{k+1})\right)\left(\xi(z_k,w_{k+1})+\tilde{\xi}(z_k,w_{k+1})\right)^\top|\cF_k}{},
\end{align*}
where $\tilde{\xi}(z_k,w_{k+1})=H(z_\kh,w_{k+1})-H(z_k,w_{k+1})$. Using Assumption~\ref{as:lipgrad}, and \eqref{eq:noisegradvar}, we have
\begin{align*}
    \expec{\norm{\tilde{\xi}(z_k,w_{k+1})}_2|\cF_k}{}\leq L_G\eta_{k+1}\expec{\norm{H(z_k,w_{k+1})}_2|\cF_k}{}\lesssim \eta_{k+1}\numberthis\label{eq:tildexibound}
\end{align*}
Then, using Assumption~\ref{as:noise}, \eqref{eq:expecconv}, and Cauchy–Schwarz inequality we get,
\begin{align*}
    \expec{\norm{{\xi}(z_k,w_{k+1})\tilde{\xi}(z_k,w_{k+1})^\top}_2}{}\lesssim \eta_{k+1}.\numberthis\label{eq:xitildexiinteraction}
\end{align*}
Then using Theorem~\ref{th:asconv}, and \eqref{eq:xitildexiinteraction}, we get,
\begin{align*}
    &\lim_{k\to\infty}\expec{\left(\xi(z_k,w_{k+1})+\tilde{\xi}(z_k,w_{k+1})\right)\left(\xi(z_k,w_{k+1})+\tilde{\xi}(z_k,w_{k+1})\right)^\top|\cF_k}{}=\Sigma.
\end{align*}
\end{proof}

\subsubsection{Additional Proofs for Section~\ref{sec:martdiffcltproof}}
\begin{proof}[Proof of Lemma~\ref{lm:noisedecomposition}]\label{pf:noisedecomposition}
Consider the following decomposition.
\begin{align*}
    \xi(z_k,w_{k+1})= H(z^*,w_{k+1})-H(z_k)+H(z^*)-H(z^*,w_{k+1})+H(z_k,w_{k+1}).
\end{align*}
Let, $\xi(z^*,w_{k+1})\coloneqq H(z^*,w_{k+1})$, and $g(z_k,w_{k+1})\coloneqq H(z^*)-H(z_k)-H(z^*,w_{k+1})+H(z_k,w_{k+1})$. Then by Assumption~\ref{as:noise}, $\xi(z^*,w_{k+1})$ is a martingale difference sequence which proves statement 1. The second statement is true by Assumption~\ref{as:asymcovar}.
To see statement 3, note that by Assumption~\ref{as:lipgrad} we have,
\begin{align*}
    \expec{\norm{g(z_k,w_{k+1})}_2^2|\cF_k}{}\leq 2(L_G+L_N)\norm{z_k-z^*}_2^2.
\end{align*}
\end{proof}
\begin{proof}[Proof of Proposition~\ref{pro:linclt}]\label{pf:linclt}
Let $m>0$ be the minimum eigen value of $Q$. Then, we have, 
$$H(z)^\top (z-z^*)\geq m\norm{z-z^*}_2^2.$$ Following similar steps to establish \eqref{eq:recursionintermed}, choosing $\eta_k\leq m/(24C)$, we have,
\begin{align*}
    &\expec{\norm{z_{k}-z^*}_2^2|\cF_{k-1}}{}\\
    \leq & \left( \left(1+3C\eta_k^2\right)\norm{z_{k-1}-z^*}_2^2-\frac{m\eta_k}{4}\norm{z_{k-1}-z^*}_2^2+C(L_G+3)\eta_k^2\right)\\
    \leq &  \left(1-\frac{m\eta_k}{8}\right)\norm{z_{k-1}-z^*}_2^2+(L_G+3)C\eta_k^2
\end{align*}
Then, for some constant $C_2>0$
\begin{align*}
    \expec{\norm{z_{k}-z^*}_2^2}{}\leq C_2\eta_k, \numberthis\label{eq:expecconv}
\end{align*}
Recall $V_k\coloneqq z_k-z^*$. Now, from \eqref{eq:fullupd}, we get,
\begin{align*}
    V_{k+1}=&V_{k}-\eta_{k+1}H(\zh,w_{k+1})\\
    =&V_{k}-\eta_{k+1}(H(\zh,w_{k+1})-H(z_k,w_{k+1}))-\eta_{k+1}H(z_k)+\eta_{k+1}\xi(z_k,w_{k+1})\\
    =&(\id-\eta_{k+1}Q)V_k-\eta_{k+1}(H(\zh,w_{k+1})-H(z_k,w_{k+1}))+\eta_{k+1}\xi(z_k,w_{k+1})\\
    =&Y_0^{k+1}V_0-\sum_{j=1}^{k+1}\eta_{j}Y_j^{k+1}(H(z_{j-1/2},w_{j})-H(z_{j-1},w_{j}))+\sum_{j=1}^{k+1}\eta_{j}Y_j^{k+1}\xi(z_{j-1},w_{j}).\numberthis\label{eq:HzklinearVkintermed}
\end{align*}
The second last equality follows using $H(z^*)=Qz^*=0$. Then,
\begin{align*}
    \sqrt{k}\bar{V}_{k}&=\underbrace{\frac{1}{\sqrt{k}}\sum_{j=1}^kY_0^k(z_0-z^*)}_{T_1}-\underbrace{\frac{1}{\sqrt{k}}\sum_{i=1}^k\sum_{j=1}^{i}\eta_jY_j^i(H(z_{j-1/2},w_{j})-H(z_{j-1},w_{j}))}_{T_2'}
    +\underbrace{\frac{1}{\sqrt{k}}\sum_{i=1}^k\sum_{j=1}^{i}\eta_jY_j^i\xi(z_{j-1},w_{j})}_{T_3}.\\\numberthis\label{eq:HzklinearVkfinal}
\end{align*}

\textbf{Bound on $\expec{\norm{T_1}_2^2}{}$}
Using Lemma~\ref{lm:sijbound}, we get
\begin{align*}
    \expec{\norm{T_1}_2^2}{}\leq \frac{1}{{k}}\norm{\sum_{j=1}^kY_0^k}_2^2\norm{z_0-z^*}_2^2\lesssim \frac{1}{{k}}.\numberthis\label{eq:T1bound}
\end{align*}

\textbf{Bound on $\expec{\norm{T_2'}_2^2}{}$}

Similar to \eqref{eq:T2bound}, we have
\begin{align*}
    \expec{\norm{T_2'}_2^2}{}\lesssim k^{1-2a}.
\end{align*}
\textbf{Bound on $\expec{\norm{T_3}_2^2}{}$}

\begin{align*}
    \frac{1}{\sqrt{k}}\sum_{i=1}^k\sum_{j=1}^{i}\eta_jY_j^i\xi(z_{j-1},w_{j})=\frac{1}{\sqrt{k}}\sum_{j=1}^k\left(\eta_j\sum_{i=j}^kY_j^i\right)\xi(z_{j-1},w_{j}).\numberthis\label{eq:Yijxi}
\end{align*}
\begin{align*}
    \eta_j\sum_{i=j}^kY_j^i&=\sum_{i=j}^k(\eta_j-\eta_{i+1})Y_j^i-Q^{-1}\left(\id-\sum_{i=j}^k\eta_{i+1}Y_j^i\right)+Q^{-1}\\
    &=\sum_{i=j}^k(\eta_j-\eta_{i+1})Y_j^i-Q^{-1}Y_j^{k+1}+Q^{-1}.\numberthis\label{eq:YijQinvclose}
\end{align*}
Plugging \eqref{eq:YijQinvclose} in \eqref{eq:Yijxi}, we get,
\begin{align*}
    &\frac{1}{\sqrt{k}}\sum_{i=1}^k\sum_{j=1}^{i}\eta_jY_j^i\xi(z_{j-1},w_{j})\\
    =&\underbrace{\frac{1}{\sqrt{k}}\sum_{j=1}^kQ^{-1}\xi(z_{j-1},w_{j})}_{I_1}
    +\underbrace{\frac{1}{\sqrt{k}}\sum_{j=1}^k\sum_{i=j}^k(\eta_j-\eta_{i+1})Y_j^i\xi(z_{j-1},w_{j})}_{I_2}
    -\underbrace{\frac{1}{\sqrt{k}}\sum_{j=1}^kQ^{-1}Y_j^{k+1}\xi(z_{j-1},w_{j})}_{I_3}.
\end{align*}
Now we show that $I_1$ converges in distribution to a multivariate normal distribution, and $I_2\to 0$, and $I_3\to 0$ in the $L_2$ sense. \vspace{0.05in}\\
\textbf{Convergence of $I_1$ to a Normal random variable in distribution}

Note that $\{\xi(z_{j-1},w_{j})\}_j$ is a martingale-difference sequence adapted to $\cF_{j-1}$. To prove the central limit theorem we need to establish Lindeberg's condition, i.e., we need to show that, 
\begin{align*}
    \lim_{\mc C\to\infty}\limsup_{j\to\infty}\expec{\norm{\xi(z_{j-1},w_{j})}_2^2\mathbbm{1}\left(\norm{\xi(z_{j-1},w_{j})}_2>\mc C\right)|\cF_{j-1}}{}\overset{P}{\to}0.
\end{align*}
Here we use the decomposition of the noise introduced in Lemma~\ref{lm:noisedecomposition}
\begin{align*}
    \xi(z_{j-1},w_{j})=\xi(z^*,w_j)+g(z_{j-1},w_j).
\end{align*}
Then,
\begin{align*}
    \mathbbm{1}\left(\norm{\xi(z_{j-1},w_{j})}_2>\mc C\right)\leq \mathbbm{1}\left(\norm{\xi(z^*,w_j)}_2>\frac{\mc C}{2}\right)+\mathbbm{1}\left(\norm{g(z_{j-1},w_j)}_2>\frac{\mc C}{2}\right).
\end{align*}
Now we have,
\begin{align*}
    &\expec{\norm{\xi(z_{j-1},w_{j})}_2^2\mathbbm{1}\left(\norm{\xi(z_{j-1},w_{j})}_2>\mc C\right)|\cF_{j-1}}{}\\
    \leq& 2\expec{\norm{\xi(z^*,w_{j})}_2^2\mathbbm{1}\left(\norm{\xi(z^*,w_{j})}_2>\frac{\mc C}{2}\right)|\cF_{j-1}}{}+
    2\expec{\norm{g(z_{j-1},w_j)}_2^2\mathbbm{1}\left(\norm{g(z_{j-1},w_j)}_2>\frac{\mc C}{2}\right)|\cF_{j-1}}{}\\
    \leq & 2\expec{\norm{\xi(z^*,w_{j})}_2^2\mathbbm{1}\left(\norm{\xi(z^*,w_{j})}_2>\frac{\mc C}{2}\right)|\cF_{j-1}}{}+
    2\norm{z_{j-1}-z^*}_2^2.
\end{align*}
The last inequality follows by part 3 of Lemma~\ref{lm:noisedecomposition}. The first term on the right hand side converges to $0$ in probability by Assumption~\ref{as:lindeberg}. For the second term, from \eqref{eq:expecconv} we have, $\norm{z_{j-1}-z^*}_2^2\overset{P}{\to}0$.
This establishes that Lindeberg's condition for Martingale CLT holds here. Then it follows by CLT for Martingales in Lemma 3.3.1 of \citep{chen2006stochastic},  
\begin{align*}
    \frac{1}{\sqrt{k}}\sum_{j=1}^kQ^{-1}\xi(z_{j-1},w_{j})\overset{D}{\to}N\left(0,Q^{-1}\Sigma Q^{-1}\right).\numberthis\label{eq:linclt}
\end{align*}
\textbf{Bound on $I_2$}

We will use the fact $j^{-a}-(j+1)^{-a}=j^{-a}o(j^{-a})$ for the following proof. 
\begin{align*}
    &\norm{\sum_{i=j}^k(\eta_j-\eta_{i+1})Y_j^i}_2\\
    =& \norm{\sum_{i=j}^k\left(\sum_{n=j}^{i}(\eta_n-\eta_{n+1})\right)Y_j^i}_2\\
    \leq & \sum_{i=j}^k\left(\sum_{n=j}^{i}n^{-a}o(n^{-a})\right)\norm{Y_j^i}_2\\
    \leq & o(j^{-a})\sum_{i=j}^k\left(\sum_{n=j}^{i}n^{-a}\right)\exp\left(-\eta\gamma\sum_{n=j}^{i}n^{-a}\right).
\end{align*}
Let $b_j^i\coloneqq\sum_{n=j}^in^{-a}$. Then,
\begin{align*}
    \norm{\sum_{i=j}^k(\eta_j-\eta_{i+1})Y_j^i}_2
    \leq  o(j^{-a})\sum_{i=j}^kb^i_j\exp\left(-\eta\gamma b_j^i\right)
    =  o(j^{-a})\sum_{i=j}^kb^i_j\exp\left(-\eta\gamma b_j^i\right)(b^i_j-b^{i-1}_j)i^a.\numberthis\label{eq:etaietajY}
\end{align*}
Note that one has the following lower bound on $b^i_j$.
\begin{align*}
    b^i_j=\sum_{n=j}^in^{-a}=j^{1-a}\sum_{n=j}^in^{-a}j^{a-1}\geq j^{1-a}\sum_{n=j}^in^{-1}\geq j^{1-a}\log\left(\frac{i}{j}\right).
\end{align*}
In other words,
\begin{align*}
    \frac{i}{j}\leq \exp(b^i_jj^{a-1}).
\end{align*}
Then,
\begin{align*}
    i^a\leq \frac{i}{j}j^a\leq \exp(b^i_jj^{a-1})j^a.
\end{align*}
Using this from \eqref{eq:etaietajY}, we have,
\begin{align*}
    \norm{\sum_{i=j}^k(\eta_j-\eta_{i+1})Y_j^i}_2
    \leq  \frac{o(j^{-a})}{j^{-a}}\sum_{i=j}^kb^i_j\exp\left(b^i_jj^{a-1}-\eta\gamma b_j^i\right)(b^i_j-b^{i-1}_j).
\end{align*}
Then, for $j\geq\ceil{\left(2/\eta\gamma\right)^{1/(1-a)}}$, we have,
\begin{align*}
    &\norm{\sum_{i=j}^k(\eta_j-\eta_{i+1})Y_j^i}_2\\
    \leq & \frac{o(j^{-a})}{j^{-a}}\sum_{i=j}^kb^i_j\exp\left(-\frac{\eta\gamma}{2} b_j^i\right)(b^i_j-b^{i-1}_j)\\
    \leq & \frac{o(j^{-a})}{j^{-a}}\int_{0}^\infty b\exp\left(-\frac{\eta\gamma}{2} b\right)db\to 0 \qquad \text{as } j\to\infty.\numberthis\label{eq:I2to0}
\end{align*}
\textbf{Bound on $I_3$}

Using Assumption~\ref{as:noise}, and \eqref{eq:expecconv}, one has, 
\begin{align*}
    \expec{\norm{\xi(z_{j-1},w_{j})}_2}{}\leq C_2,
\end{align*}
for some constant $C_2>0$. Now, using $Q\succcurlyeq 0$, \eqref{eq:noisevar}, \eqref{eq:expecconv}, and Lemma~\ref{lm:Yijbound}, we have,
\begin{align*}
    &\expec{\norm{\frac{1}{\sqrt{k}}\sum_{j=1}^kQ^{-1}Y_j^{k+1}\xi(z_{j-1},w_{j})}_2^2}{}\\
    \leq & \frac{1}{k}\sum_{j=1}^k\norm{Q^{-1}}_2^2\norm{Y_j^{k+1}}_2^2\expec{\norm{\xi(z_{j-1},w_{j})}_2^2}{}\\
    \lesssim & \frac{1}{k}\sum_{j=1}^k\norm{Y_j^{k+1}}_2^2\\
    \lesssim & \frac{1}{k}\sum_{j=1}^k\exp\left(-\frac{2\gamma\eta}{1-a}\left((k+2)^a-(j+1)^a\right)\right)\\
    \lesssim & \exp\left(-\frac{2\gamma\eta}{1-a}(k+2)^a\right)\frac{1}{k}\int_{1}^k\exp\left(\frac{2\gamma\eta}{1-a}(j+1)^a\right)dj\\
    =& \frac{\exp\left(-\frac{2\gamma\eta}{1-a}(k+2)^a\right)}{a\left(\frac{2\gamma\eta}{1-a}\right)^\frac{1}{a}}\frac{1}{k}\int_{\frac{2\gamma\eta}{1-a}2^a}^{\frac{2\gamma\eta}{1-a}(k+1)^a}e^uu^\frac{1-a}{a}dj\\
    \leq & \frac{\exp\left(-\frac{2\gamma\eta}{1-a}(k+2)^a\right)}{ka\left(\frac{2\gamma\eta}{1-a}\right)}\exp\left(\frac{2\gamma\eta}{1-a}(k+1)^a\right)(k+1)^{1-a}\\
    \lesssim & k^{-a}.\numberthis\label{eq:QYxibound}
\end{align*}
The second last inequality follows by integration by parts.
Combining \eqref{eq:T1bound}, \eqref{eq:T2bound}, \eqref{eq:linclt}, \eqref{eq:I2to0}, and \eqref{eq:QYxibound}, we have Proposition~\ref{pro:linclt}.
\end{proof}

\section{Proof of Section~\ref{sec:markov}}
\begin{proof}[Proof of Claim~\ref{cl:tailprob}]\label{pf:tailprob}
Let, 
\begin{align*}
    \tilde{\xi}(z_k,w_{k+1})=H(z_k,w_{k+1})-H(z_\kh,w_{k+1}).
\end{align*}
Let $\mathscr{V}_M$ denote the level set of $\mathcal{L}(z)$, i.e., 
\begin{align*}
    \mathscr{V}_M\coloneqq \left\lbrace z\in\mc Z|\mathcal{L}(z)\leq M\right\rbrace.
\end{align*}
Define the following terms. For any $\delta>0$, and some $M>0$,
\begin{align*}
&\Gamma(\mathcal{K} )\coloneqq \inf\{k\geq 1|z_{k}\notin \mathcal{K} \}, \quad \upsilon(\boldsymbol{d})\coloneqq \inf\{k\geq 1| \norm{z_{k+1}-z_{k}}_2\geq d_k\}\\ &\Upsilon(\mc K,\boldsymbol{d})\coloneqq\min(\Gamma(\mathcal{K} ),\upsilon(\boldsymbol{d})),\\
&S_{l,n}(\boldsymbol{\eta},\boldsymbol{d},\mathcal{K})\coloneqq\mathbbm{1}(\Upsilon(\mathcal{K},\boldsymbol{d})\geq n)\sum_{k=l}^n\eta_k\left(\xi(z_{k-1},w_{k})+\tilde{\xi}(z_{k-1},w_{k})\right),\numberthis\label{eq:Slndef}\\
&\mc A(\delta,\boldsymbol{d},\boldsymbol{\eta}, M)\coloneqq \left\lbrace\mathbb{P}_{z_0,w_1}\left(\sup_{k\geq 1}\abs{S_{1,k}(\boldsymbol{\eta},\boldsymbol{d},\mathscr{V}_M)}>\delta\right)+\mathbb{P}_{z_0,w_1}\left(\Gamma(\mathscr{V}_M )>\upsilon(\boldsymbol{d})\right)\right\rbrace.\numberthis\label{eq:calAdef}
\end{align*}
After one reinitialization, for a given $\delta>0$, and  $M>0$, $\mc A(\delta,\boldsymbol{d},\boldsymbol{\eta}, M)$ provides an upper bound on the probability that another reinitialization happens. The next lemma from \citep{andrieu2005stability} relates the $\mathbb{P}_{z_0,w_1}\left(k_{\varkappa_{\mc M}}<\infty\right)$ with $\mc A(\delta,\boldsymbol{d},\boldsymbol{\eta}, M)$.
\begin{lemma}[Proposition 4.2, Corollary 4.3, \citep{andrieu2005stability}]
    Let Assumption~\ref{as:strongcon}, and Assumption~\ref{as:noisemar} be true, and $\mathcal{K}\subset \mathscr{V}_{M_0}$ for some large enough $M_0>0$. Then, for any $M>M_0$, there exists a constant $\delta_0>0$ such that, for any $\varkappa_{n}>q_0$, we have
    \begin{align*}
        \mathbb{P}_{z_0,w_1}\left(k_{\varkappa_{n}}<\infty\right)\leq \prod_{l=q_0}^{{\varkappa_{n}}-1}\sup_{q\geq l}\mc A(\delta_0,\boldsymbol{d}^{\leftarrow q},\boldsymbol{\eta}^{\leftarrow q}, M),
    \end{align*}
    where $\boldsymbol{d}^{\leftarrow q}$, and $\boldsymbol{\eta}^{\leftarrow q}$ represent the sequences $\{d_{k+q}\}_k$, and $\{\eta_{k+q}\}_k$ respectively. In other words, for any $q_1\geq q_0$, there exists a constant $C>0$, such that, for any $q_2\geq Q_1$, we have,
    \begin{align*}
    \mathbb{P}_{z_0,w_1}\left(\sup_{k\geq 1}T_k\geq q_2\right)\leq C\left(\sup_{q\geq q_1}\mc A(\delta_0,\boldsymbol{d}^{\leftarrow q},\boldsymbol{\eta}^{\leftarrow q}, M)\right)^{q_2}
    \end{align*}
\end{lemma}
The above lemma suggests that if we can show that $\mc A(\delta_0,\boldsymbol{d}^{\leftarrow q},\boldsymbol{\eta}^{\leftarrow q}, M)$ converges to $0$ when $q\to\infty$ then Claim~\ref{cl:tailprob} is true. To do so, we show that $\mathbb{P}_{z_0,w_1}\left(\sup_{k\geq l}\abs{S_{l,k}(\boldsymbol{\eta},\boldsymbol{d},\mathcal{K})}>\delta\right)$, and $\mathbb{P}_{z_0,w_1}\left(\Gamma(\mathcal{K} )>\upsilon(\boldsymbol{d})\right)$ are bounded for any compact set $\mathcal{K}$, $\delta>0$, and positive integer $l$. This is analogous to Proposition 5.2 of \citep{andrieu2005stability} but for us the Algorithm and the noise sequence are different. The following bound combined with Assumption~\ref{as:stepsize} shows $\mc A(\delta_0,\boldsymbol{d}^{\leftarrow q},\boldsymbol{\eta}^{\leftarrow q}, M)\to 0$ as $q\to\infty$.
\begin{lemma}\label{lm:claimprooffinal}
    Assume $w_1$ be such that $\mc V(w_1)<\infty$. Let Assumption~\ref{as:strongcon},  Assumption~\ref{as:noisemar}, and Assumption~\ref{as:stepsize} be true. Then we have,
    \begin{align*}
    &\mathbb{P}_{z_0,w_1}^{\boldsymbol{\eta}}\left(\sup_{n\geq l}\norm{S_{l,n}(\boldsymbol{d},\boldsymbol{\eta},\mathcal{K})}_2\geq \delta\right)\leq C\delta^{-{\alpha_0}}\left\lbrace \left(\sum_{k=l}^\infty\eta_k^2\right)^{{\alpha_0}}+\left(\sum_{k=l}^\infty\eta_kd_k\right)^{{\alpha_0}}\right\rbrace,\numberthis\label{eq:Snlbound}\\
    &\mathbb{P}_{z_0,w_1}\left(\Gamma(\mathcal{K} )>\upsilon(\boldsymbol{d})\right)\leq C\sum_{k=1}^\infty\left(\eta_kd_k^{-1}\right)^{\alpha_0}.\numberthis\label{eq:stepsizelargelater}
\end{align*}
\end{lemma}
We defer the proof of Lemma~\ref{lm:claimprooffinal} to Appendix~\ref{pf:claimprooffinal}.
This completes the proof of Claim~\ref{cl:tailprob}, and consequently, Theorem~\ref{th:finitetrunc}.
\end{proof}
\subsection{Proof of Lemma~\ref{lm:claimprooffinal}}
\begin{proof}[Proof of Lemma~\ref{lm:claimprooffinal}]\label{pf:claimprooffinal}

Let $\tilde{\xi}(z_k,w_{k+1})\coloneqq H(z_k,w_{k+1})-H(z_\kh,w_{k+1})$. Then,
\begin{align*}
    &\mathbb{P}_{z_0,w_1}^{\boldsymbol{\eta}}\left(\sup_{n\geq l}\norm{S_{l,n}(\boldsymbol{d},\boldsymbol{\eta},\mathcal{K})}_2\geq \delta\right)\\
    \leq &\mathbb{P}_{z_0,w_1}^{\boldsymbol{\eta}}\left(\sup_{n\geq l}\norm{\mathbbm{1}(\Upsilon(\mathcal{K},\boldsymbol{d})\geq n)\sum_{k=l}^n\eta_k\xi(z_{k-1},w_{k})}_2+\sup_{n\geq l}\norm{\mathbbm{1}(\Upsilon(\mathcal{K},\boldsymbol{d})\geq n)\sum_{k=l}^n\eta_k\tilde{\xi}(z_{k-1},w_{k})}_2\geq \delta\right)\\
    \leq &\mathbb{P}_{z_0,w_1}^{\boldsymbol{\eta}}\left(\sup_{n\geq l}\norm{\mathbbm{1}(\Upsilon(\mathcal{K},\boldsymbol{d})\geq n)\sum_{k=l}^n\eta_k\xi(z_{k-1},w_{k})}_2\geq \frac{\delta}{2}\right)\\
    &+\mathbb{P}_{z_0,w_1}^{\boldsymbol{\eta}}\left(\sup_{n\geq l}\norm{\mathbbm{1}(\Upsilon(\mathcal{K},\boldsymbol{d})\geq n)\sum_{k=l}^n\eta_k\tilde{\xi}(z_{k-1},w_{k})}_2\geq \frac{\delta}{2}\right).\numberthis\label{eq:Psnlintermed}
\end{align*}
By Proposition 5.2 of \citep{andrieu2005stability}, we have,
\begin{align*}
    &\mathbb{P}_{z_0,w_1}^{\boldsymbol{\eta}}\left(\sup_{n\geq l}\norm{\mathbbm{1}(\Upsilon(\mathcal{K},\boldsymbol{d})\geq n)\sum_{k=l}^n\eta_k\xi(z_{k-1},w_{k})}_2\geq \frac{\delta}{2}\right)\\
    \leq & C\delta^{-\alpha_0}\left\lbrace \left(\sum_{k=l}^\infty\eta_k^2\right)^{\alpha_0/2}+\left(\sum_{k=l}^\infty\eta_kd_k^{\alpha_0}\right)^{{\alpha_0}}\right\rbrace\mathcal{V}^{\alpha_0}(w_0).\numberthis\label{eq:boundfromverstable}
\end{align*}
We now establish the bound on the second term of \eqref{eq:Psnlintermed}. Note that under Assumption~\ref{as:lipgrad}, we have,
\begin{align*}
    \norm{\tilde{\xi}(z_k,w_{k+1})}_2=\norm{H(z_k,w_{k+1})-H(z_\kh,w_{k+1})}_2\leq L_G\eta_{k+1}\norm{H(z_k,w_{k+1})}_2.\numberthis\label{eq:tildexinorm}
\end{align*}
Then we have,
\begin{align*}
    &\mathbb{P}_{z_0,w_1}^{\boldsymbol{\eta}}\left(\sup_{n\geq l}\norm{\mathbbm{1}(\Upsilon(\mathcal{K},\boldsymbol{d})\geq n)\sum_{k=l}^n\eta_k\tilde{\xi}(z_{k-1},w_{k})}_2\geq \frac{\delta}{2}\right)\\
    \leq & \mathbb{P}_{z_0,w_1}^{\boldsymbol{\eta}}\left(\sup_{n\geq l}\mathbbm{1}(\Upsilon(\mathcal{K},\boldsymbol{d})\geq n)\sum_{k=l}^n\eta_k\norm{\tilde{\xi}(z_{k-1},w_{k})}_2\geq \frac{\delta}{2}\right)\\
     \leq & \mathbb{P}_{z_0,w_1}^{\boldsymbol{\eta}}\left(\sup_{n\geq l}\sum_{k=l}^n\eta_k^2\norm{H(z_{k-1},w_{k})}_2\mathbbm{1}(\Upsilon(\mathcal{K},\boldsymbol{d})\geq k)\geq \frac{\delta}{2}\right)\\
     \leq & \mathbb{P}_{z_0,w_1}^{\boldsymbol{\eta}}\left(\left(\sum_{k=l}^\infty\eta_k^2\norm{H(z_{k-1},w_{k})}_2\mathbbm{1}(\Upsilon(\mathcal{K},\boldsymbol{d})\geq k)\right)^{\alpha_0}\geq \left(\frac{\delta}{2}\right)^{\alpha_0}\right)\\
     \leq & \left(\frac{\delta}{2}\right)^{-\alpha_0}\expec{\left(\sum_{k=l}^\infty\eta_k^2\norm{H(z_{k-1},w_{k})}_2\mathbbm{1}(\Upsilon(\mathcal{K},\boldsymbol{d})\geq k)\right)^{\alpha_0}}{}\numberthis\label{eq:probnormHsum}
\end{align*}
The last inequality follows from Using Markov's inequality. 
Now note that, for constants $p_1+p_2=\alpha_0$, $p_1,p_2=0,1,\cdots,\alpha_0$, using Assumption~\ref{as:noisemar}, and Holder's inequality for $k_1,k_2\geq l$ we have,
\begin{align*}
    &\expec{\left(\norm{H(z_{k_1-1},w_{k_1})}_2\mathbbm{1}(\Upsilon(\mathcal{K},\boldsymbol{d})\geq k_1)\right)^{p_1}\left(\norm{H(z_{k_2-1},w_{k_2})}_2\mathbbm{1}(\Upsilon(\mathcal{K},\boldsymbol{d})\geq k_2)\right)^{p_2}}{}\\
    \leq & \expec{\left(\norm{H(z_{k_1-1},w_{k_1})}_2\mathbbm{1}(\Upsilon(\mathcal{K},\boldsymbol{d})\geq k_1)\right)^{\alpha_0}}{}^{\frac{p_1}{\alpha_0}}
    \expec{\left(\norm{H(z_{k_2-1},w_{k_2})}_2\mathbbm{1}(\Upsilon(\mathcal{K},\boldsymbol{d})\geq k_2)\right)^{\alpha_0}}{}^{\frac{p_2}{\alpha_0}}\\
    \lesssim &\mathcal{V}(w_1)^{\alpha_0}.\numberthis\label{eq:Hinteractionbound}
\end{align*}
Combining, \eqref{eq:probnormHsum}, and \eqref{eq:Hinteractionbound}, we get, 
\begin{align*}
\mathbb{P}_{z_0,w_1}^{\boldsymbol{\eta}}\left(\sup_{n\geq l}\norm{\mathbbm{1}(\Upsilon(\mathcal{K},\boldsymbol{d})\geq n)\sum_{k=l}^n\eta_k\tilde{\xi}(z_{k-1},w_{k})}_2\geq \frac{\delta}{2}\right)\leq \left(\frac{\delta}{2}\right)^{-\alpha_0}\left(\sum_{k=l}^\infty\eta_k^2\right)^{\alpha_0}\mathcal{V}(w_0)^{\alpha_0}.\numberthis\label{eq:tildexibound}
\end{align*}
Combining \eqref{eq:boundfromverstable}, and \eqref{eq:tildexibound}, we get \eqref{eq:Snlbound}.
Now we show \eqref{eq:stepsizelargelater}. Using \eqref{eq:tildexinorm}, Assumption~\ref{as:noisemar}, and Markov's inequality, we have
\begin{align*}
&\mathbb{P}_{z_0,w_1}\left(\Gamma(\mathcal{K} )>\upsilon(\boldsymbol{d})\right)\\
\leq& \mathbb{P}_{z_0,w_1}\left(\Gamma(\mathcal{K} )\geq\upsilon(\boldsymbol{d})\right)\\
=& \sum_{k=1}^\infty \mathbb{P}_{z_0,w_1}\left(\upsilon(\boldsymbol{d})=k,\Gamma(\mathcal{K} )\geq k\right)\\
=&\sum_{k=1}^\infty \mathbb{P}_{z_0,w_1}\left(\norm{H(z_k,w_{k+1})+\tilde{\xi}(z_k,w_{k+1})}_2\geq (d_k\eta_k^{-1}),\upsilon(\boldsymbol{d})=k,\Gamma(\mathcal{K} )\geq k\right)\\
\lesssim &\sum_{k=1}^\infty \left(\frac{1+\eta_{k+1}}{(d_k\eta_k^{-1})}\right)^{\alpha_0}\\
\lesssim & \sum_{k=1}^\infty (d_k^{-1}\eta_k)^{\alpha_0}
\end{align*}
\end{proof}
\subsection{Proof of Theorem~\ref{th:asconvmar}}
Before stating the proof we state the following results which we need for the proof. Assumption~\ref{as:noisemar} implies the following result from \citep{liang2010trajectory}. 
\begin{lemma}[\cite{liang2010trajectory}]\label{lm:poisregular}
Let Assumption~\ref{as:noisemar} be true. Then we have the following:
\begin{enumerate}
    \item For any $z\in\mc Z$, the Markov kernel $P_{z}$ has a single stationary distribution $\pi_{z}$. Moreover, $H({z},w):{z}\times\mathbb{R}^{d_\t+d_\m}\to{\mc Z}$ is measurable for all ${z}\in{\mc Z}$, $\expec{H({z},w)}{w\sim\pi_{z}}<\infty$.
    \item For any ${z}\in{\mc Z}$, the Poisson equation $u({z},w)-P_{z} u({z},w)=H({z},w)-H({z})$ has a solution $u({z},w)$, where $P_{z} u({z},w)=\int_{\mathbb{R}^{d_\t+d_\m}}u({z},w')P_{z} (w,w')dw'$. There exist a function $\mathcal{V}:\mathbb{R}^{d_\t+d_\m}\to[1,\infty)$ such that for all ${z}\in{\mc Z}$, the following holds:
    \begin{enumerate}
        \item $\sup_{{z}\in\mc Z}\norm{H({z},w)}_\mathcal{V}<\infty$,
        \item $\sup_{{z}\in\mc Z}\left(\norm{u({z},w)}_\mathcal{V}+\norm{P_{z} u({z},w)}_\mathcal{V}\right)<\infty$,
        \item $\sup_{{z}\in\mc Z}\left(\norm{u({z},w)-u({z}',w)}_\mathcal{V}+\norm{P_{z} u({z},w)-P_{{z}'} u({z}',w)}_\mathcal{V}\right)<\norm{{z}-{z}'}_2$.
    \end{enumerate}
\end{enumerate}
\end{lemma}
\begin{lemma}[\citep{liang2010trajectory}]\label{lm:noisedecompboundapp}
Let Assumption~\ref{as:strongcon}, Assumption~\ref{as:noisemar}, and Assumption~\ref{as:stepsize} be true. Then the following decomposition takes place
\begin{align*}
    \xi(z_{k-1},w_{k})=e_{k}+\nu_{k}+\zeta_{k},
\end{align*}
where, $\{e_k\}$ is martingale difference sequence, $\norm{\nu_k}_V\leq \eta_{k}$, and $\zeta_k=(\tilde{\zeta}_k-\tilde{\zeta}_{k+1})/\eta_k$, 
where $\expec{\|\tilde{\zeta}_k\|_2}{}\leq \eta_k$. Specifically,
\begin{align*}
	 &e_{k+1}=u(z_k,w_{k+1})-P_{z_k}u(z_k,w_{k})\\
	 &\nu_{k+1}=P_{z_{k+1}}u(z_{k+1},w_{k+1})-P_{z_{k}}u(z_{k},w_{k+1})+\frac{\eta_{k+2}-\eta_{k+1}}{\eta_{k+1}}P_{z_{k+1}}u(z_{k+1},w_{k+1})\\
	 &\tilde{\zeta}_{k+1}=\eta_{k+1}P_{z_{k}}u(z_{k},w_{k})\\
	 &\zeta_{k+1}=\frac{\tilde{\zeta}_{k+1}-\tilde{\zeta}_{k+2}}{\eta_{k+1}}. \numberthis\label{eq:compdef}
	\end{align*}
 Moreover, when $z_k\overset{a.s.}{\to} z^*$, one has the following,
 \begin{align*}
     \frac{1}{\sqrt{n}}\sum_{k=1}^ne_ke_k^\top\overset{d}{\to}N(0,\Sigma_s),\numberthis\label{eq:ekclt}
 \end{align*}
 where $\Sigma_s=\lim_{k\to\infty}e_ke_k^\top$. 
\end{lemma}
\begin{proof}[Proof of Theorem~\ref{th:asconvmar}]\label{pf:asconvmar}
Recall the definition of the following perturbed sequence.
\begin{align*}
    \tz_{k+1}=&z_{k+1}+\tzeta_{k+2}\\
    =&\tz_k-\eta_{k+1} H(\tz_k)+\eta_{k+1}(e_{k+1}+\tau_{k+1}),
\end{align*}
where $\tau_{k+1}=\nu_{k+1}+\tilde\xi(z_k,w_{k+1})+\omega_{k+1}$, $\tilde{\xi}(z_k,w_{k+1})=H(z_\kh,w_{k+1})-H(z_k,w_{k+1})$, and $\omega_{k+1}=H(\tz_k)-H(z_k)$. Then, using Lemma~\ref{lm:noisedecompboundapp}, Assumption~\ref{as:lipgrad}, \eqref{eq:noisegradvar}, and Theorem~\ref{th:asconvmar}, we have, 
\begin{align*}
    \expec{\norm{\tau_k}_2^2}{}\lesssim &\expec{\norm{\nu_{k}}_2^2}{}+\expec{\norm{\tilde\xi(z_k,w_{k+1})}_2^2}{}+\expec{\norm{\omega_{k}}_2^2}{}\\
    \lesssim &\eta_{k+1}^2+\norm{z_\kh-z_k}_2^2+\norm{\tilde{z}_k-z_k}_2^2\\
    \lesssim & \eta_{k+1}^2.\numberthis\label{eq:tauklbound}
\end{align*}
Using Assumption~\ref{as:strongcon}, Assumption~\ref{as:lipgrad}, and choosing $\eta$ appropriately small, we have,
\begin{align*}
   &\norm{\tz_{k+1}-z^*}_2^2\\
   =&\norm{\tz_{k}-z^*-\eta_{k+1} H(\tz_k)+\eta_{k+1}(e_{k+1}+\tau_{k+1})}_2^2 \\
   =&\norm{\tz_{k}-z^*}_2^2+\eta_{k+1}^2\norm{H(\tilde{z}_k)}_2^2-2\eta_{k+1}H(\tz_k)^\top (\tz_k-z^*)+2\eta_{k+1}^2\left(\norm{e_{k+1}}_2^2+\norm{\tau_{k+1}}_2^2\right)\\
   &+2\eta_{k+1}(\tz_{k}-z^*-\eta_{k+1} H(\tz_k))^\top (e_{k+1}+\tau_{k+1})\\
   \leq & \left(1+L_G^2\eta_{k+1}^2\right)\norm{\tz_k-z^*}_2^2-2\eta_{k+1}\mc G(\tilde{z}_k)+2\eta_{k+1}^2\left(\norm{e_{k+1}}_2^2+\norm{\tau_{k+1}}_2^2\right)\\
     &+2\eta_{k+1}(\tz_{k}-z^*-\eta_{k+1} H(\tz_k))^\top e_{k+1}+2\eta_{k+1}^2\norm{\tz_k-z^*}_2^2+2\eta_{k+1}^4\norm{H(\tz_k)}_2^2+\norm{\tau_{k+1}}_2^2\\
     \leq & \left(1+(L_G^2+4)\eta_{k+1}^2\right)\norm{\tz_k-z^*}_2^2-2\eta_{k+1}\mc G(\tilde{z}_k)+2\eta_{k+1}^2\norm{e_{k+1}}_2^2+3\norm{\tau_{k+1}}_2^2
     +2\eta_{k+1}(\tz_{k}-z^*-\eta_{k+1} H(\tz_k))^\top e_{k+1}.
\end{align*}
Taking expectation on both sides conditional on $\cF_k$, using the fact that $\{e_k\}_k$ is a Martingale-difference sequence, and \eqref{eq:tauklbound}, we get, 
\begin{align*}
   \expec{ \norm{\tz_{k+1}-z^*}_2^2|\cF_k}{}\leq &\left(1+(L_G^2+4)\eta_{k+1}^2\right)\norm{\tz_k-z^*}_2^2-2\eta_{k+1}\mc G(\tilde{z}_k)+2\eta_{k+1}^2\expec{\norm{e_{k+1}}_2^2|\cF_k}{}+3\expec{\norm{\tau_{k+1}}_2^2|\cF_k}{}\\
   \lesssim & \left(1+(L_G^2+4)\eta_{k+1}^2\right)\norm{\tz_k-z^*}_2^2-2\eta_{k+1}\mc G(\tilde{z}_k)+\eta_{k+1}^2.\numberthis\label{eq:zktildeevol}
\end{align*}
Recall that this is exactly of the same form as \eqref{eq:recursionintermed}. Then following a similar procedure as the proof of Theorem~\ref{th:asconv}, we get, 
\begin{align*}
    \tz_k\overset{a.s.}{\to} z^*.\numberthis\label{eq:tildezas}
\end{align*}
For any $\vartheta>0$, using Lemma~\ref{lm:noisedecompbound} we have,
\begin{align*}
    \sum_{k=1}^\infty P\left(\norm{\tilde{\zeta}_k}_2>\vartheta\right)=\sum_{k=1}^\infty P\left(\norm{\tilde{\zeta}_k}_2^2>\vartheta^2\right)\leq \sum_{k=1}^\infty\frac{\expec{\norm{\tilde{\zeta}_k}_2^2}{}}{\vartheta^2}\leq \sum_{k=1}^\infty\frac{\eta_k^2}{\vartheta^2}<\infty.
\end{align*}
Then by Borel-Cantelli lemma,
\begin{align*}
    \tilde{\zeta}_k\overset{a.s.}{\to} 0. \numberthis\label{eq:tildezetaas}
\end{align*}
Combining \eqref{eq:tildezas}, and \eqref{eq:tildezetaas}, we have,
\begin{align*}
    z_k\overset{a.s.}{\to} z^*.\numberthis\label{eq:zasmar}
\end{align*}
Following similar methods as in Theorem~\ref{th:asconv}, we get, 
\begin{align*}
    \bar{z}_k\overset{a.s.}{\to} z^*.\numberthis\label{eq:zasmar}
\end{align*}

Similar to \eqref{eq:indicatorbound}, using \eqref{eq:zktildeevol} we have, 
\begin{align*}
    \expec{\norm{\tz_{k}-z^*}_2^2\mathbbm{1}(\mathcal{E}_{\mc D}>k)}{}\leq C_1\eta_k, \numberthis\label{eq:indicatorboundtilde}
\end{align*}
Then, using Lemma~\ref{lm:noisedecompbound}, we get,
\begin{align*}
    \expec{\norm{z_k-z^*}_2^2\mathbbm{1}(\mathcal{E}_{\mc D}>k)}{}\lesssim \eta_{k}.
\end{align*}
\end{proof}
\subsection{Proof of Theorem~\ref{th:marclt}}
\begin{proof}[Proof of Theorem~\ref{th:marclt}]\label{pf:marclt}

Using Lemma~\ref{lm:noisedecompbound}, we have,
\begin{align*}
    \expec{\norm{\sqrt{n}(\bar{\tilde{z}}_n-\bar{{z}}_n)}_2^2}{}=\expec{\norm{\frac{1}{\sqrt{n}}\sum_{k=1}^n(\tilde{z}_k-z_k)}_2^2}{}\leq \expec{\frac{1}{{n}}\sum_{k=1}^n\sum_{j=1}^n\norm{\tzeta_{k+1}}_2\norm{\tzeta_{j+1}}_2}{}\lesssim n^{1-2a}. \numberthis\label{eq:tzkzkasymequi}
\end{align*}
Then, $\sqrt{n}(\bar{{z}}_n-z^*)$ has the same asymptotic distribution as $\sqrt{n}(\bar{\tilde{z}}_n-z^*)$. So it is sufficient to establish the asymptotic normality of $\sqrt{n}(\bar{\tilde{z}}_n-z^*)$.

Define $\tV_k=\tz_k-z^*$. Then, similar to \eqref{eq:HzklinearVkintermednonlin}, we have,
\begin{align*}
    \tV_{k+1}=&(\id-\eta_{k+1}Q^*)\tV_k+\eta_{k+1}e_{k+1}+\eta_{k+1}\tau_{k+1}+\eta_{k+1}(H(\tz_{k})-Q^*\tV_{k})\\
    =& Y_0^{k+1}\tV_0+\sum_{j=1}^{k+1}\eta_{j}Y_j^{k+1}e_j+\sum_{j=1}^{k+1}\eta_{j}Y_j^{k+1}\tau_j+\sum_{j=1}^{k+1}\eta_jY_j^{k+1}(H(\tz_{j-1})-Q^*\tV_{j-1}).
\end{align*}
Then,
\begin{align*}
    \bar{\tV}_{k}=& \frac{1}{\sqrt{k}}\sum_{j=1}^kY_0^{k+1}\tV_0+\frac{1}{\sqrt{k}}\sum_{i=1}^k\sum_{j=1}^i\eta_{j}Y_j^ie_j+\frac{1}{\sqrt{k}}\sum_{i=1}^k\sum_{j=1}^{i}\eta_{j}Y_j^{i}\tau_j +\frac{1}{\sqrt{k}}\sum_{i=1}^k\sum_{j=1}^{i}\eta_jY_j^i(H(\tz_{j-1})-Q^*\tV_{j-1}).
\end{align*}
Similar to \eqref{eq:T1bound}, \eqref{eq:T2bound}, and \eqref{eq:T4bound}, we have,
\begin{align*}
&\expec{\norm{\frac{1}{\sqrt{k}}\sum_{j=1}^kY_0^{k+1}\tV_0}_2^2}{}\lesssim k^{-1}\\
    &\expec{\norm{\frac{1}{\sqrt{k}}\sum_{i=1}^k\sum_{j=1}^{i}\eta_{j}Y_j^{i}\tau_j}_2^2}{}\lesssim k^{1-2a}\\
    &\expec{\norm{\frac{1}{\sqrt{k}}\sum_{i=1}^k\sum_{j=1}^{i}\eta_jY_j^i(H(\tz_{j-1})-Q^*\tV_{j-1})}_2}{}\to 0.
\end{align*}
Similar to the proof of \eqref{eq:linclt}, combining, \eqref{eq:ekclt}, and \eqref{eq:zasmar}, we get,
\begin{align*}
    \frac{1}{\sqrt{k}}\sum_{i=1}^k\sum_{j=1}^i\eta_{j}Y_j^ie_j\overset{D}{\to}N\left(0,{Q^*}^{-1}\Sigma_s {Q^*}^{-1}\right).\numberthis\label{eq:finalcltmar}
\end{align*}
The proof of the asymptotic optimality of the covariance is similar to the martingale difference case and hence we omit it here. 
\end{proof}
\end{document}